\documentclass[12pt,twoside]{amsart}
\usepackage{amssymb}
\usepackage{amscd}
\usepackage{hyperref}
\usepackage{color}

\usepackage[normalem]{ulem} 
\newcommand\redout{\bgroup\markoverwith
{\textcolor{red}{\rule[.5ex]{2pt}{0.4pt}}}\ULon}

\title[Plt threefolds with non-normal centres]
{Purely log terminal threefolds with non-normal centres in characteristic two} 
\author{Paolo Cascini and Hiromu Tanaka} 
\address{Department of Mathematics, Imperial College, London, 180 Queen's Gate, 
London SW7 2AZ, UK} 
\email{p.cascini@imperial.ac.uk}
\address{Graduate School of Mathematical Sciences, 
The University of Tokyo, 
3-8-1 Komaba, Meguro-ku, Tokyo 153-8914, JAPAN} 
\email{tanaka@ms.u-tokyo.ac.jp}
\thanks{Both of the authors were funded by EPSRC}

\newcommand{\Sing}[0]{{\operatorname{Sing}}}
\newcommand{\red}[0]{{\operatorname{red}}}
\newcommand{\Coker}[0]{{\operatorname{Coker}}}
\newcommand{\Ker}[0]{{\operatorname{Ker}}}

\newcommand{\Proj}[0]{{\operatorname{Proj}}}
\newcommand{\Spec}[0]{{\operatorname{Spec}}}

\newcommand{\Supp}[0]{{\operatorname{Supp}}}
\newcommand{\Pic}[0]{{\operatorname{Pic}}}

\newcommand{\Ex}[0]{{\operatorname{Ex}}}
\newtheorem{thm}{Theorem}[section]
\newtheorem{lem}[thm]{Lemma}
\newtheorem{cor}[thm]{Corollary}
\newtheorem{prop}[thm]{Proposition}

\theoremstyle{definition}

\newtheorem{dfn}[thm]{Definition}

\newtheorem{rem}[thm]{Remark}
      
\newtheorem*{claim}{Claim}
\newtheorem{nota}[thm]{Notation}         

\newtheorem{nothing}[thm]{}
\newtheorem{assumption}[thm]{Assumption}

\makeatletter
  
  \@addtoreset{equation}{section}
  \makeatother

\newcommand{\MO}{\mathcal{O}}
\newcommand{\R}{\mathbb{R}}
\newcommand{\Q}{\mathbb{Q}}
\newcommand{\Z}{\mathbb{Z}}

\begin{document}

\maketitle

\begin{abstract}
We show that many classical results of the minimal model program  
do not hold over an algebraically closed field of characteristic two. Indeed, 
we construct a three dimensional plt pair
whose codimension one part is not normal, 
a three dimensional klt singularity which is not rational nor Cohen-Macaulay, 
and a klt Fano threefold  with non-trivial intermediate cohomology. 
\end{abstract}

\tableofcontents
\setcounter{section}{0}

\section{Introduction}
 Many important results  in birational geometry  rely on the study of singularities that appear in the minimal model program. In particular, thanks to the work of Shokurov, purely log terminal (or plt for short) pairs have played a crucial role in the proof of the existence of flips (e.g. see \cite{shokurov03,hm10,bchm10,hx13}). 
A basic but important property for plt pairs $(X, \Delta)$ in characteristic zero 
is that the coefficient one part $\llcorner \Delta\lrcorner$ is normal \cite[Proposition 5.51]{km98}. 
Its proof heavily depends on  Kawamata--Viehweg vanishing theorem, 
which fails in positive characteristic. 
Thus, it is natural to ask whether the same property holds in positive characteristic. 
For example, if $X$ is either a surface or a threefold in characteristic $p>5$, 
then the question holds  true \cite[Theorem 3.1 and Proposition 4.1]{hx13}. 

The purpose of this paper is to give a negative answer in characteristic two. 

\begin{thm}\label{intro-plt}
Let $k$ be an algebraically closed field in characteristic two. 
Then there exists a $\Q$-factorial three dimensional plt pair $(Z, E)$ over $k$ 
such that $E$ is a prime divisor which is not normal. 
\end{thm}

The main idea of Theorem \ref{intro-plt} is to apply a cone-like construction 
to some klt del Pezzo surfaces which violate Kawamata--Viehweg vanishing 
(cf. Subsection \ref{ss-sketch}). 

The normality of the codimension one part of a plt pair 
is closely related to the extension problem of log pluri-canonical sections. 
Using the same construction as in the proof of Theorem~\ref{intro-plt}, 
we  produce a plt pair such that 
the restriction maps for the log pluri-canonical divisors are not surjective, contrary to what happens in characteristic zero
(cf. \cite[Theorem~5.4.21]{hm07}, \cite[Corollary~1.8]{dhp13}):

\begin{thm}\label{intro-extension}
Let $k$ be an algebraically closed field of characteristic two. 
Then there exist a three dimensional affine variety $Z$ over $k$ and 
a projective birational morphism $g\colon Y \to Z$ 
which satisfies the following properties:  
\begin{enumerate}
\item $(Y, E+B)$ is a plt pair such that $\llcorner E+B \lrcorner=E$, 
\item $K_{Y}+E+B$ is semi-ample, and 
\item there exists a positive integer $m_0$ such that the restriction map 
$$H^0(Y, \MO_{Y}(mm_0(K_{Y}+E+B))) \to 
H^0(E, \MO_{Y}(mm_0(K_Y+E+B))|_{E})$$
is not surjective for any positive integer $m$.  
\end{enumerate}
\end{thm}

Finally, using similar methods, 
we construct three dimensional Kawamata log terminal (or klt for short) singularities 
which are not rational nor Cohen--Macaulay (see \cite[Theorem 5.22 and Corollary 5.25]{km98} for the corresponding result in characteristic zero) and a
three dimensional klt Fano variety $X$ with $H^2(X,\MO_X) \neq 0$:

\begin{thm}\label{intro-non-rational}
Let $k$ be an algebraically closed field of characteristic two. 
Then there exists a three dimensional klt variety $Z$ over $k$ such that
\begin{enumerate}
\item $Z$ is not Cohen--Macaulay, and 
\item there is  a projective birational morphism 
$h\colon W \to Z$ from a smooth threefold $W$ such that $R^1h_*\MO_W \neq 0$.
\end{enumerate}
\end{thm}

\begin{thm}\label{intro-bad-Fano}
Let $k$ be an algebraically closed field of characteristic two. 
Let $r$ be a positive integer. 
Then there exists 
a three dimensional projective $\Q$-factorial klt variety $Z$ over $k$ 
such that $-K_Z$ is ample, $\rho(Z)=1$ and  $\dim_k H^2(Z, \MO_Z)=r$. 
\end{thm}

\medskip

\subsection{Sketch of the proof}\label{ss-sketch}
We  briefly overview some of the  ideas used  in the proof of Theorem~\ref{intro-plt}. 
To  this end, we first describe a method to obtain a plt pair whose codimension one part is not normal.  
Assume that 
there exist a smooth Fano variety $T$ with $\rho(T)=1$, an ample divisor $A$ and 
a normal prime divisor $E$ on $T$ such that 
\begin{enumerate}
\item $(T, E)$ is plt and $-(K_T+E)$ is ample, 
\item $H^1(T, \MO_T(nA))=0$ for any $n\geq 0$, 
\item $H^1(T, \MO_T(A-E))\neq 0$, and  
\item $H^2(T, \MO_T(mA-E))=0$ for any $m \geq 2$. 
\end{enumerate}
Note, in particular, that Kawamata--Viehweg vanishing  fails for $T$. Let $\pi_Y\colon Y:=\mathbb P_T(\MO_T \oplus \MO_T(A)) \to T$ 
and let $g\colon Y \to Z$  be the birational contraction of the section $T^-$ 
of $\pi_Y$ whose normal bundle is anti-ample. Thus, $Z$ is a cone over $T$. 
Let $E^Y:=\pi_Y^*(E)$ and $E^Z :=g_*E^Y$. 
Assuming  inversion of adjunction,  (1) implies that $(Z, E^Z)$ is plt. 
We now  show that $E^Z$ is not normal. 
Consider the exact sequence
\[
g_*\MO_Y \xrightarrow{\rho} g_*\MO_{E^Y} \to R^1g_*\MO_Y(-E^Y) \to R^1g_*\MO_Y.
\]
In order to show that $E^Z$ is not normal, 
it is enough to prove that $\rho$ is not surjective. 
Therefore, we would like to show that 
$R^1g_*\MO_Y(-E^Y) \neq 0$ and $R^1g_*\MO_Y=0$. 
Using  Serre vanishing theorem, both of these claims follow 
from the assumptions (2)--(4) (cf.  proof of Theorem~\ref{t-plt-main}). 

\medskip

\subsubsection{Construction} 
As explained above, we would like to find a Fano variety $T$
satisfying the above properties (1)--(4). 
As far as the authors know, 
only three families of klt Fano varieties are known to violate Kawamata--Viehweg vanishing \cite{LR97, schroer07, maddock16}.
However, none of them seem to be sufficient to our purpose. 

Therefore, we need to find a new Fano variety satisfying the properties (1)--(4) above. 
If we admit some singularities, then 
such examples can be constructed by taking some divisors 
on a family of klt del Pezzo surfaces in characteristic two. 
These surfaces were introduced by Keel and M\textsuperscript cKernan in \cite{km99}. 
More specifically, we can find a klt del Pezzo surface $T$, 
an ample $\Z$-divisor $A$ and a prime divisor $E$ on $T$ 
which satisfy the properties (1)--(4) 
(cf. Section~\ref{s-KM-surface} and Lemma~\ref{l-plt-compute}). 
Unfortunately, our divisor $A$ is not Cartier, 
which makes our construction more complicated. 
Indeed, we introduce a new cone-like construction to apply the same idea as above. 

The construction is as follows. Let $\psi\colon S \to T$ be the minimal resolution. 
 We consider the graded $\MO_S$-algebra 
$\mathcal A=\bigoplus_{m=0}^{\infty}\mathcal A_m$ 
whose graded pieces $\mathcal A_m$ are defined by 
\[\begin{aligned}
\mathcal A_0&:=\MO_S\\
\mathcal A_1&:=\MO_S \oplus \MO_S(\psi^*A)\\
\mathcal A_2&:=\MO_S \oplus \MO_S(\psi^*A) \oplus \MO_S(2\psi^*A)\\
&\dots
\end{aligned}
\]
and we equip $\mathcal A$ with the canonical multiplication structure. 
Note that the fractional part $\{\psi^*A\}$ is not zero in our case. 
Let $X:=\Proj_S(\mathcal A)$. 
By contracting some divisors on $X$, 
we get a threefold $Y$ which admits a $\mathbb P^1$-fibration $\pi_Y\colon Y\to T$. Such a fibration corresponds to the $\mathbb P^1$-bundle
$\mathbb P_T(\MO_T \oplus \MO_T(A))$ over $T$ at the beginning of Subsection~\ref{ss-sketch}, 
and it is a special case of a Seifert bundle (see \cite[Section 9.3]{kollar13} and the reference therein).

Finally, as above, we consider a birational morphism $g\colon Y\to Z$ which contracts a section $T^-$ of $\pi_Y$. 
Section~\ref{s-cone} is devoted to construct these threefolds and  check some of the properties we need. 
Note that, even though $S$ is smooth, $X$ is a singular variety. On the other hand, 
since we impose some assumptions on $\{\psi^*A\}$ (cf. Assumption~\ref{a-A}), 
we can show that $X$ and $Y$ are klt 
by constructing an explicit resolution of singularities. 
In Section \ref{s-examples}, we prove our main results.

\medskip
\textbf{Acknowledgement: } 
We would like to thank Y. Gongyo, C. D. Hacon, 
Y. Kawamata, 
J. Koll\'ar,
J. M\textsuperscript cKernan, S. Takagi, and C. Xu for many useful discussions and comments. The first author
would like to thank the National Center for Theoretical Sciences in Taipei
and Professor J.A. Chen for their generous hospitality, where some of the work for this paper was completed. We are also grateful to the referees for carefully reading the paper and for many useful comments.

\section{Preliminaries}

\subsection{Notation}

Throughout this paper, 
we work over an algebraically closed field $k$ of characteristic $p>0$. 
From Section~\ref{s-KM-surface}, we assume that $p=2$. 
We say that $X$ is a {\em variety} 
if $X$ is an integral  scheme which is separated and of finite type over $k$. 
A {\em curve} (resp. {\em surface}, resp. {\em threefold}) 
is a variety of dimension one (resp. two, resp. three). 
Given a variety $X$, $\Pic(X)$ denotes the {\em Picard group} of $X$ 
and $\Pic (X)_{\Q}:=\Pic(X) \otimes_{\Z} \Q$. 
Given an $\mathbb F_p$-scheme $X$, 
we denote by $F\colon X \to X$ the absolute Frobenius morphism. 
Given a scheme $X$, the {\em reduced part} $X_{\red}$ of $X$ 
is the reduced closed subscheme of $X$ 
such that the induced closed immersion $X_{\red} \to X$ is surjective. 

Given a proper morphism $f\colon X\to Y$ between normal varieties, we say that  two 
$\mathbb Q$-Cartier $\mathbb Q$-divisors 
$D_1,D_2$ on $X$ are {\em numerically equivalent over} $Y$, denoted $D_1\equiv_f D_2$, if their difference is numerically trivial on any fibre of $f$. 
We denote by $\rho(X/Y)$ the {\em relative Picard number} 
and we set $\rho(X):=\rho(X/\Spec\,k)$. 
If $f$ is a birational morphism, $\Ex(f)$ denotes the {\em exceptional locus} of $f$. 

We refer to \cite[Section 2.3]{km98} or \cite[Definition~2.8]{kollar13} for the classical definitions of singularities 
(e.g., {\em klt, plt, log canonical}) appearing in the minimal model program. 
Note  that we always assume that 
for any klt (resp.\ plt, log canonical) pair $(X, \Delta)$, 
the $\Q$-divisor $\Delta$ is effective. 

Given a $\Q$-divisor $D$ on a normal variety $X$ and an open subset $U$ of $X$, 
we define 
\[
\Gamma(U, \MO_X(D):=\{\varphi \in K(X)\,|\, (\text{div}(\varphi)+D)|_U \geq 0\}.
\]
It follows that $\MO_X(D)=\MO_X(\llcorner D\lrcorner)$. 

Given positive integers $m_1, \cdots, m_n$, we define 
$\mathbb P_k(m_1, \cdots, m_n):=\Proj\,k[x_1, \cdots , x_n]$, 
where $k[x_1, \cdots, x_n]$ is the graded polynomial ring such that 
 $x_i$ is a homogeneous element of degree $m_i$, for $i=1,\dots,n$. 
Note  that $\mathbb P_k(1, m)$ is isomorphic to $\mathbb P^1_k$ 
for any positive integer $m$. 

\subsection{$\Q$-factoriality}

We now describe a criterion for $\Q$-factoriality. 
Note that its proof is valid only in positive characteristic 
as it relies on \cite{keel99}. 

\begin{lem}\label{l-Qfac-criterion}
Let $f\colon X \to Y$ be a birational morphism of projective normal varieties 
such that $E:=\Ex(f)$ is a prime divisor. 
Let $g\colon E \to f(E)$ be the restriction of $f$ along $E$. 
Assume that 
\begin{enumerate}
\item[(a)] $X$ is $\Q$-factorial,
\item[(b)] $\rho(E/f(E))=1$, and
\item[(c)] any $g$-numerically trivial Cartier divisor $M$ on $E$ is $g$-semi-ample. 
\end{enumerate}
Then the following hold:
\begin{enumerate}
\item For any curve $\zeta$ on $X$ such that $f(\zeta)$ is a point, 
the sequence 
\[
0 \to \Pic (Y)_{\Q} \xrightarrow{f^*} \Pic(X)_{\Q} \xrightarrow{\cdot \zeta} \Q \to 0
\]
is exact. 
\item $\rho(Y)=\rho(X)-1$.
\item $Y$ is $\Q$-factorial. 
\end{enumerate}
\end{lem}

\begin{proof}
We first show (1). 
Clearly $f^*$ is injective, 
the composite map $(\cdot \zeta) \circ f^*$ is zero, and  
the map  $\cdot \zeta$ is surjective. 
Let $N$ be a Cartier divisor on $X$ such that $N \cdot \zeta=0$. 
It follows from (b) that $N|_E \equiv_g 0$. 
Therefore, we get $N\equiv_f 0$. 
Let $H$ be an ample Cartier divisor on $Y$. 

\medskip 

We first prove the following 

\begin{claim}\label{c-globally-nef}
There exists a positive integer $m$ such that $N+mf^*H$ is nef. 
\end{claim}

Consider the  set 
\[
I:=\{S\,|\,S\text{ is an integral closed subscheme of }X\text{ such that  }
S\not\subset \Ex(f)\}.
\]
If $S\in I$, then the induced morphism $S \to f(S)$ is birational. 
Therefore, for every $S\in I$, 
we can find $n_S\in\mathbb Z_{>0}$ such that $(N+n_Sf^*H)|_{S}$ is big. 
Let $n_1:=n_X$. 
By Kodaira's lemma, we may write 
$N+n_1f^*H=A+D$ where $A$ is an ample $\mathbb Q$-divisor and $D$ is an effective $\mathbb Q$-divisor on $X$. 
Let $D=\sum e_jD_j$ be the  decomposition into irreducible components and let 
 $n_2:=\max_{D_j\in I}\{n_1, n_{D_j}\}$. 
For any $D_j\in I$, we denote by $D_j^N$ its normalisation and  we apply Kodaira's lemma to $(N+n_2f^*H)|_{D_j^N}$. By proceeding as above, since
at each step the dimension drops,  after finitely many steps we may find $n\in\mathbb Z_{>0}$ such that 
$(N+nf^*H)|_{S}$ is big for every $S\in I$. 
In particular, we have that $(N+nf^*H)\cdot C>0$ for every curve $C$ on $X$ with $C\not\subset \Ex(f)$. 

Since $N|_E$ is $g$-numerically trivial, 
after possibly replacing $N$ and $H$ by $\ell N$ and $\ell H$ respectively 
for some positive integer $\ell$,
(c) implies that there exists a divisor $Q$ on $f(E)$ such that $N|_E=g^*Q$. Thus,  for any sufficiently large positive integer  $m$, we have that  $(N+mf^*H)|_E=g^*(Q+mH|_{f(E)})$ is nef and  the Claim follows. 

\medskip

By the Claim, it follows that if  $m$ is a sufficiently large positive integer, the divisor $N+mf^*H$ is nef and big. If  $B$ is a curve on $X$, then  $(N+mf^*H) \cdot B=0$  if and only if $f(B)$ is a point. 
By (c) and Keel's theorem \cite[Theorem 0.2]{keel99}, 
it follows that $N+mf^*H$ is semi-ample. 
Since it induces the same morphism as $f$, 
there exists a $\Q$-Cartier $\Q$-divisor $F$ on $Y$ such that 
$N+mf^*H=f^*F$. Thus, (1) holds. 

Note that (2) follows from (1). We now show (3). Let $A_Y$ be an ample divisor on $Y$. By applying Kodaira's lemma to 
$f^*A_Y$, it follows that $E \cdot \zeta<0$.
Let $D$ be a prime divisor on $Y$ and let $D_X$ be its proper transform on $X$. 
By (1), there exists $\alpha \in \Q$ such that $D_X+\alpha E \equiv_f 0$. 
By (1), there exists a $\Q$-Cartier $\Q$-divisor $F$ on $Y$ such that 
$D_X+\alpha E \sim_{\Q} f^*F$, hence $D$ is $\Q$-Cartier. 
Thus, (3) holds. 
\end{proof}

\section{Some vanishing criteria}

In this section, we establish two vanishing criteria: 
Proposition~\ref{p-eff-nef-big} and Proposition~\ref{p-birat-kvv}. 
These results will be used in Section~\ref{s-cone} and Section~\ref{s-examples}.

\subsection{Vanishing for effective nef and big divisors}

The purpose of this subsection is to show Proposition~\ref{p-eff-nef-big} 
which is a special case of Kawamata--Viehweg vanishing. 
The key result is Lemma~\ref{l-frob-inje}. 
We first recall a basic result about Serre duality:

\begin{lem}\label{l-CM}
Let $X$ be a projective normal surface. 
Let $D$ be a $\Z$-divisor on $X$. 
Then there is an isomorphism of $k$-vector spaces 
\[
H^i(X, \MO_X(D)) \simeq H^{2-i}(X, \MO_X(K_X-D))^*,
\]
for any  $i=0,1$ and $2$, 
where $H^*$ denotes the dual vector space of $H$. 
\end{lem}

\begin{proof}
The claim follows from \cite[Theorem~5.71]{km98} and 
the fact that $\MO_X(D)$ is $S_2$, hence Cohen--Macaulay. 
\end{proof}

\begin{lem}\label{l-frob-inje}
Let $X$ be a proper normal variety and 
let $D$ be an effective $\Z$-divisor. 
If $H^1(X, \MO_X)=0$, then the natural map 
\[
H^1(X, \MO_X(-D)) \to H^1(X, F_*\MO_X(-pD)),
\]
induced by the absolute Frobenius morphism $F\colon X \to X$ is injective. 
\end{lem}

\begin{proof}
Consider the closed immersion $i\colon D\hookrightarrow X$ and the composition
\[
F'\colon pD\hookrightarrow X \xrightarrow {F} X. 
\]
Then, we obtain the commutative diagram 

\[
\begin{CD}
0 @>>> \MO_X(-D) @>>> \MO_X @>>> i_*\MO_D @>>> 0\\
@. @VV\alpha V @VV\beta V @VV\gamma V\\
0 @>>> F_*(\MO_X(-pD)) @>>> F_*\MO_X @>>> F'_*(\MO_{pD}) @>>> 0.\\
\end{CD}
\]

\medskip

We first prove the following

\begin{claim}
The natural homomorphism 
$\gamma\colon i_*\MO_D \to F'_*\MO_{pD}$
is injective.
\end{claim}

Fix an affine open subset $U$ of $X$. 
Take $f \in \Gamma(U, \MO_X)$ such that $\beta(f)=f^p \in 
\Gamma(U, F_*(\MO_X(-pD)))=\Gamma(U, \MO_X(-pD))$. 
This implies that 
\[
p~{\rm div}(f)={\rm div}(f^p) \geq pD.
\]
Thus, ${\rm div}(f) \geq D$ and the Claim holds. 

\medskip

Since $H^1(X, \MO_X)=0$ and $H^0(\beta)$ is bijective, 
it follows that 
\[
H^0(X, \Coker(\beta))=0.
\]
By the Claim and the snake lemma, we get an injection: 
\[
\Coker(\alpha) \hookrightarrow \Coker(\beta).
\]
Therefore, we obtain $H^0(X, \Coker(\alpha))=0$, as desired. 
\end{proof}

\begin{prop}\label{p-eff-nef-big}
Let $X$ be a projective klt rational surface. 
If $D$ is an effective nef and big $\Z$-divisor, then 
\[H^1(X, \MO_X(-D))=0 \quad\text{and}\quad  H^1(X, \MO_X(K_X+D))=0.
\]
\end{prop}

\begin{proof}
By Lemma~\ref{l-CM}, it is enough to show the first equality. 
Since $H^1(X, \MO_X)=0$ \cite[Theorem 5.4 and Remark 5.5]{tanaka12}, 
we can apply Lemma~\ref{l-frob-inje}. 
Thus, there exists an injection 
\[
H^1(X, \MO_X(-D))\hookrightarrow H^1(X, \MO_X(-p^eD))
\]
for any positive integer $e$. 
By 
Lemma~\ref{l-CM}, it is enough to show that 
\[
H^1(X, \MO_X(K_X+p^eD))=0
\]
for some positive integer $e$. Let $a$ be a positive integer such that the Cartier index of $p^aD$ is not divisible by $p$. 
Then, for any
sufficiently large positive integer $e$ such that $e-a>0$ and $e-a$ is sufficiently divisible, the divisor $(p^e-p^a)D$ is Cartier. Thus, \cite[Theorem 2.11]{tanaka12a} implies the claim. 
\end{proof}

\begin{rem}
Note that in Proposition~\ref{p-eff-nef-big}, the assumption that $D$ is an effective divisor is needed as otherwise Kawamata-Viehweg vanishing might fail (cf. Theorem~\ref{t-KM-h1}).
\end{rem}

\subsection{A birational Kawamata--Viehweg vanishing}

We now describe a sufficient condition 
to apply the birational Kamawata--Viehweg vanishing (Proposition~\ref{p-birat-kvv}). 
This can be considered as a higher dimensional version 
of the birational Kawamata--Viehweg vanishing between surfaces 
proved by Koll\'ar and Kov\'acs (\cite[Theorem~10.4]{kollar13}). 

\begin{dfn}
Let $f\colon X \to S$ be a projective morphism from a smooth variety $X$ to a variety $S$. 
Let $\Delta$ be an effective simple normal crossing $\Q$-divisor on $X$ whose coefficients are at most one. 
We say that {\em Kawamata--Viehweg vanishing} (or {\em KVV} for short) 
{\em holds for} $(f, \Delta)$ if 
$R^if_*\MO_X(D)=$ for any $i>0$ and any Cartier divisor $D$ on $X$ such that $D-(K_X+\Delta)$ is $f$-ample. 
\end{dfn}

\begin{prop}\label{p-birat-kvv}
Let $f\colon X \to Y$ be a projective birational morphism from a smooth variety $X$ to a normal variety $Y$. 
Assume that 
\begin{enumerate}
\item[(a)] $E$ is an effective simple normal crossing divisor  
such that $\Supp E=\Ex(f)$ and $-E$ is $f$-ample. 
\item[(b)] For any $f$-exceptional prime divisor $E_i$ and any
effective simple normal crossing $\Q$-divisor $\Gamma$ on $E_i$ 
whose coefficients are at most one, 
 KVV holds for $(f \circ \iota_j, \Gamma)$, 
where $\iota_j:E_j \hookrightarrow X$ is the inclusion. 
\item [(c)]  $\Delta$ is an effective $\Q$-divisor on $X$ 
whose coefficients are at most one and $\Supp \Delta \cup \Ex(f)$ is simple normal crossing.
\end{enumerate}
 
Then  KVV holds for $(f, \Delta)$. 
\end{prop}

\begin{proof}
Let $E=\sum_{i \in I} e_iE_i$ be the  decomposition into irreducible components with $e_i \in \Z_{>0}$. 
Let $D$ be a Cartier divisor on $X$ such that $D-(K_X+\Delta)$ is $f$-ample. 
We want to show that $R^if_*\MO_X(D)=0$ for $i>0$. 

Consider the decomposition 
\[
\Delta=B+\Delta_0,
\]
where $B$ and $\Delta_0$ are effective $\Q$-divisors 
such that $\Supp \Delta_0 \subset \Ex(f)$ and  
any irreducible component of $B$ is not contained in $\Ex(f)$. 
Let 
\[
A:=D-(K_X+\Delta)=D-(K_X+B+\Delta_0).
\]
We define 
\[
\lambda_0:=0 \quad \text{and}\quad D_0:=D.
\]
We construct a sequence of triples 
\[
\{(\lambda_j, D_j, \Delta_j)\}_{j \in \Z_{>0}}
\]
which satisfies the following properties
\begin{enumerate}
\item[$(1)_j$] $\lambda_j$ is a rational number such that 
\[
0\leq \lambda_0 \leq \lambda_1 \leq  \cdots \leq \lambda_j.
\]
\item[$(2)_j$] $\Delta_j$ is an effective $f$-exceptional $\Q$-divisor whose coefficients are at most one. 
\item[$(3)_j$] $D_j$ is a $\Z$-divisor such that $D_j-(K_X+B+\Delta_j)=A-\lambda_jE.$ 
\item[$(4)_j$] If $R^qf_*\MO_X(D_j)=0$ for $q>0$, then $R^qf_*\MO_X(D_{j-1})=0$ for $q>0$. 
\item[$(5)$] $\lim_{j \to \infty} \lambda_j=\infty$. 
\end{enumerate}
Fix $j \in \Z_{\geq 0}$ and assume that we have already constructed 
\[
(\lambda_0, D_0, \Delta_0), \cdots, (\lambda_j, D_j, \Delta_j)
\]
such that $(1)_k, (2)_k, (3)_k, (4)_k$ hold for $0 \leq k \leq j$. 
Let $\mu_j$ be the non-negative rational number such that 
\[
\Delta_j+\mu_j E \leq 1
\]
and the coefficient of some prime divisor $E_{i_j}$ in $\Delta_j+\mu_j E$ is equal to one. 
We define 
\[
\lambda_{j+1}:=\lambda_j+\mu_j, \quad \Delta_{j+1}:=\Delta_j+\mu_j E-E_{i_j}, \quad D_{j+1}:=D_j-E_{i_j}.
\]
Then $(1)_{j+1}$ and $(2)_{j+1}$ hold. 
We see that $D_{j+1}$ is a $\Z$-divsior such that 
\[
\begin{aligned}
D_{j+1}-(K_X+B+\Delta_{j+1})&=(D_j-E_{i_j})-(K_X+B+\Delta_j+\mu_j E-E_{i_j})\\
&=(D_j-(K_X+B+\Delta_j))-\mu_j E\\
&=A-\lambda_j E-\mu_j E=A-\lambda_{j+1} E.
\end{aligned}
\]
Thus, $(3)_{j+1}$ holds. 
We now show $(4)_{j+1}$. 
Consider the exact sequence
\[
0 \to \MO_X(D_{j+1}) \to \MO_X(D_j) \to \MO_{E_{i_j}}(D_j) \to 0.
\]
Since $\Supp (B+\Delta_{j+1})$ does not contain $E_{i_j}$ and 
\[
D_j-(K_X+B+E_{i_j}+\Delta_{j+1})=A-\lambda_{j+1}E
\]
is $f$-ample, (b) implies  that 
\[
R^qf_*\MO_{E_{i_j}}(D_j)=0
\]
for any $q>0$. 
Therefore, $(4)_{j+1}$ holds. 

We now show $(5)$. 
There exists an infinite increasing sequence of positive integers $\{j_k\}_{k \in \Z_{>0}}$ such that $i_{j_k}$ is constant. After possibly reordering, we may assume $i_{j_k}=1$ for all $k$. 
For all $j$, let $\delta_j$ be the coefficient of $\Delta_j$ along $E_1$. 
For any $k$, it follows from the construction above that $\delta_{j_k+1}=0$ and
\[
\sum_{i=j_{k}+1}^{j_{k+1}} \mu_i e_1 = 1. 
\]
Thus, 
\[
\sum_{i=0}^{\infty} \mu_i=\infty
\]
and since $\lambda_j=\sum_{i=0}^{j-1} \mu_i$, (5) holds. 

\medskip

By $(4)_j$, it is enough to find $j>0$ such that 
\[
R^qf_*\MO_X(D_j)=0
\]
for any $q>0$. 
This follows from $(3)_j$, (5) and the relative Fujita vanishing theorem \cite[Theorem 1.5]{keeler03}.
\end{proof}

\section{Keel--M\textsuperscript cKernan surfaces}\label{s-KM-surface}

In this section, we first recall the construction of surfaces obtained by Keel and M\textsuperscript cKernan (Subsection~\ref{ss-KM-surface}). 
After that, we produce some divisors on these surfaces 
which violate Kawamata--Viehweg vanishing (Theorem~\ref{t-KM-h1}).

\subsection{Construction}\label{ss-KM-surface}
We now recall the construction of rank one del Pezzo surfaces in characteristic two, obtained by Keel and M\textsuperscript cKernan \cite[page 77]{km99}. 

Let $k$ be an algebraically closed field of characteristic two. Let
$\Gamma_0$ be a strange conic in  $\mathbb P_{k}^2$ \cite[Example IV.3.8.2]{Hartshorne77}, i.e. 
there exists a closed point $Q\in \mathbb P^2_k$ such that any line passing through $Q$ is tangent to $\Gamma_0$. 
Let $S_1 \to \mathbb P_k^2$ be the blow-up of $\mathbb P^2_{k}$ at the point $Q$. 
Then $S_1$ admits a $\mathbb P_k^1$-bundle structure $\rho_1\colon S_1 \to \mathbb P_k^1$. 
If $\Gamma_1$ is the inverse image of $\Gamma_0$ in $S_1$,  then 
the induced morphism $\Gamma_1 \hookrightarrow S_1 \xrightarrow{\rho_1} \mathbb P_k^1$ 
is  purely inseparable  of degree two. 

Let $P_1, \dots, P_d\in \Gamma_1$ be distinct points and let $S_2\to S_1$ be the blow-up of $S_1$ at  the points $P_1, \dots, P_d$.
Since the fibre $F_i:=\rho_1^{-1}(\rho_1(P_i))$ is tangent to $\Gamma_1$
for  $i=1,\dots,d$, the proper transforms $F'_i$ and $\Gamma_2$ of $F_i$ and $\Gamma_1$ respectively,  intersect in a point $Q'_i$.  Let $S\to S_2$ be the blow-up of $S_2$ at the points $Q'_1,\dots,Q'_d$ and 
let $\Gamma$ to be the proper transform of $\Gamma_2$ on $S$. 
Then the fibration 
$\rho\colon S \to \mathbb P_k^1$
induced by $\rho_1$ has exactly $d$ singular fibres: 
\[
2E_1+\ell_1+\ell'_1, \dots, 2E_d+\ell_d+\ell'_d,
\]
where, for each $i=1,\dots,d$,  $E_i$ is a $(-1)$-curve,   $\ell_i$ is the proper transform of $F_i$ and $\ell'_i$ is the proper  transform of the exceptional divisor of $S_2\to S_1$ with centre $P_i$. 
In particular, $\ell_i^2=\ell_i'^2=-2$ and $\Gamma^2=4-2d$. We will always assume that $d\ge 3$, so that $\Gamma^2<0$. 

Let $\varphi\colon S \to \mathbb P^2_k$ be the induced morphism and let 
\[
\psi\colon S \to T
\]
be the birational contraction of the curves $\Gamma, \ell_1, \ell'_1, \dots, \ell_d, \ell'_d$. 
For each $i=1,\dots,d$, let  $E^T_i:=\psi_*(E_i)$. Let $F$ be the general fibre of $\rho\colon S \to \mathbb P^1_k$. 
Note that $\Gamma\cdot F=2$.

\begin{lem}\label{l_ST}
With the same notation as above, the following  hold:
\begin{enumerate}
\item $-K_S \sim \Gamma+F$. 
\item $T$ is klt and $\rho(T)=1$.
\item $-K_T$ is ample and $-K_T\sim 2E_i^T$ for any $i=1,\dots,d$.  
\end{enumerate}
\end{lem}

\begin{proof} (1) and (2) follow from the construction. 
Since $F\sim 2E_i+\ell_i+\ell'_i$ for each $i=1,\dots,d$, we have that (3) is an immediate consequence of (1) and (2). 
\end{proof}

\subsection{Counterexamples to Kawamata--Viehweg vanishing}

We now construct a sequence of divisors on $T$ 
which violate Kawamata--Viehweg vanishing. 
Note that, by \cite[Theorem 1.2] {CTW15b}, Kawamata--Viehweg vanishing  holds 
on a klt del Pezzo surface of sufficiently large characteristic. On the other hand, in 
\cite{CT16b}, we show that Kawamata--Viehweg vanishing fails over smooth rational surfaces in 
arbitrary  characteristic.  

\begin{thm}\label{t-KM-h1}
We use the same notation as in Subsection~\ref{ss-KM-surface}. 
Let
\[
A:=\sum_{i=1}^{q_1}E^T_i-\sum_{j=q_1+1}^{q_1+q_2}E^T_j
\]
for some non-negative integers $q_1$ and $q_2$ such that $q_1+q_2 \leq d$. 
 
Then the following  hold: 
\begin{enumerate}
\item If $q_2>0$, then $H^0(T, \MO_T(A))=0$. 
\item $H^2(T, \MO_T(A))=0$. 
\item $\chi(T, \MO_T(A))=1-q_2+(q_1-q_2-d+3)\llcorner \frac{q_1-q_2}{2d-4}\lrcorner-
\left(\llcorner \frac{q_1-q_2}{2d-4}\lrcorner\right)^2(d-2).$
\item If $q_2=0$, then $H^1(T, \MO_T(A))=0$. 
\item If $q_2>0$ and $q_1-q_2 \geq 0$, then $h^1(T, \MO_T(A))=q_2-1$. 
\item If $q_2>0$ and $q_1-q_2 < 0$, then $h^1(T, \MO_T(A))=q_1$. 
\end{enumerate}
\end{thm}

\begin{proof}
Since 
\begin{eqnarray*}
\psi^*A
&=&\sum_{i=1}^{q_1}(E_i+\frac{1}{2}(\ell_i+\ell_i'))-\sum_{j=q_1+1}^{q_1+q_2}(E_j+\frac{1}{2}(\ell_j+\ell_j'))
+\frac{q_1-q_2}{2d-4}\Gamma,
\end{eqnarray*}
we have that 
\begin{eqnarray*}
\llcorner \psi^*A\lrcorner
&=&\sum_{i=1}^{q_1}E_i-\sum_{j=q_1+1}^{q_1+q_2}(E_j+\ell_j+\ell_j')
+\llcorner \frac{q_1-q_2}{2d-4}\lrcorner\Gamma\\
&\equiv& \frac{q_1-q_2}{2}F-\frac{1}{2}\sum_{i=1}^{q_1+q_2}(\ell_i+\ell'_i)
+\llcorner \frac{q_1-q_2}{2d-4}\lrcorner\Gamma.\\
\end{eqnarray*}

Since $K_S+\frac{2d-6}{2d-4}\Gamma$ is $\psi$-numerically trivial 
by (1) of Lemma \ref{l_ST}, 
we have that 
$\llcorner \psi^*A\lrcorner -(K_S+\frac{2d-6}{2d-4}\Gamma)$ is $\psi$-nef. 
We have that 
\[
\psi_*\MO_S(\llcorner \psi^*A\lrcorner)=\MO_T(A)
\]
and by  Kawamata--Viehweg vanishing theorem for  birational morphisms between surfaces  \cite[Theorem 10.4]{kollar13}, we have that

\[
 R^i\psi_*\MO_S(\llcorner \psi^*A\lrcorner)=0 \qquad\text{for any $i>0$}. 
 \]
Thus, the Leray spectral sequence induces an isomorphism  
\[
H^i(S, \MO_S(\llcorner \psi^*A\lrcorner)) \simeq H^i(T, \MO_T(A)) \qquad\text{for any $i\ge 0$}.
\]
In particular, $\chi(S, \MO_S(\llcorner \psi^*A\lrcorner))=\chi(T, \MO_T(A))$.

\medskip

We first show (1). 
Since $q_2>0$, we have that $\llcorner \frac{q_1-q_2}{2d-4}\lrcorner \leq 0$ and $H^0(E_i, \MO_{E_i}(\llcorner \psi^*A\lrcorner))=0$ for any $i=1, \dots, q_1$.
From the exact sequence
\[
0 \to \MO_S(\llcorner \psi^*A\lrcorner -\sum_{i=1}^{q_1}E_i) \to \MO_S(\llcorner \psi^*A\lrcorner) 
\to \bigoplus_{i=1}^{q_1}\MO_{E_i}(\llcorner \psi^*A\lrcorner) \to 0,
\]
 it follows that 
\[
H^0(S,  \MO_S(\llcorner \psi^*A\lrcorner -\sum_{i=1}^{q_1}E_i)) \simeq H^0(S, \MO_S(\llcorner \psi^*A\lrcorner)).
\]
Since $q_2>0$, we have that 
$H^0(S,  \llcorner \psi^*A\lrcorner -\sum_{i=1}^{q_1}E_i)=0$. 
Thus, (1) holds.

We now show (2). 
By Lemma \ref{l-CM}, we have that 
\[
h^2(T, \MO_T(A)) = h^0(T, \MO_T(K_T-A)).
\]
By (3) of Lemma \ref{l_ST},  it follows that  (1) implies (2).

We now show (3). 
We have 
\begin{eqnarray*}
(\llcorner \psi^*A\lrcorner)^2
&=&-(q_1+q_2)+2(q_1-q_2)\llcorner \frac{q_1-q_2}{2d-4}\lrcorner+
\left(\llcorner \frac{q_1-q_2}{2d-4}\lrcorner\right)^2(4-2d).
\end{eqnarray*}
and by (1) of Lemma \ref{l_ST},
\[
(\llcorner \psi^*A\lrcorner) \cdot (-K_S)
=(6-2d)\llcorner \frac{q_1-q_2}{2d-4}\lrcorner+(q_1-q_2).
\]
Thus,  Riemann--Roch implies (3).

Proposition~\ref{p-eff-nef-big} and (3) of Lemma \ref{l_ST} imply (4). 
(5) follows from (1), (2), (3) and the fact that if $q_2>0$ and $q_1\geq q_2$ then  $\llcorner \frac{q_1-q_2}{2d-4}\lrcorner=0$. 

Finally, we show (6). Assume that $q_2>0$ and $q_1<q_2$. 
If $d= q_2=3$, then $q_1=0$ and  (3) implies that $\chi(S,\MO_S(\llcorner \psi^*A\lrcorner))=0$. 
Otherwise, we have that $\llcorner \frac{q_1-q_2}{2d-4}\lrcorner=-1$ and (3) implies 
$\chi(S, \MO_S(\llcorner \psi^*A\lrcorner))=-q_1.$

Thus, in both cases, (1) and (2) imply (6). 
\end{proof}

\section{A cone construction}\label{s-cone}
The goal of this Section is to construct a cone over a del Pezzo surface, associated to an ample divisor which is not necessarily Cartier. Note that, for simplicity, although our construction works in higher generality, we only consider a special case which is needed to prove our main results.

Throughout this Section, we use the same notation as in 
Subsection~\ref{ss-KM-surface}. 
We fix an ample $\Z$-divisor $A$ on $T$ which satisfy the following 

\begin{assumption}\label{a-A}
There exist a $\Z$-divisor $B$ on  $S$ and positive integers $m_C$ 
such that 
\[
\psi^*A=B+\sum_{C \subset \Ex(\psi)} \frac{1}{m_C}C.
\]
\end{assumption}

In Section \ref{s-examples},
we will construct explicit examples of divisors $A$ on $T$  satisfying Assumption \ref{a-A}.

\subsection{A generalisation of $\mathbb P^1$-bundles}

Let 
\[
\mathcal A:=\bigoplus_{m=0}^{\infty}\mathcal A_m
\]
be the quasi-coherent graded $\MO_S$-algebra defined by 
\[
\begin{aligned}
\mathcal A_0&:=\MO_S\\
\mathcal A_1&:=\MO_S \oplus \MO_S(\psi^*A)\\
\mathcal A_2&:=\MO_S \oplus \MO_S(\psi^*A) \oplus \MO_S(2\psi^*A)\\
&\cdots \\
\mathcal A_m &:=\MO_S \oplus \MO_S(\psi^*A) \oplus \cdots \oplus \MO_S(m\psi^*A)\\
&\cdots
\end{aligned}
\]
where we equip $\mathcal A$ with the canonical multiplication. 
We define 
\[
\pi\colon X:=\Proj_S\,\mathcal A \to S,
\]
as in \cite[(3.1.3)]{egaii} (note that in \cite{egaii}, the notation $\Proj~\mathcal A$ is used instead of $\Proj_S~\mathcal A$). Recall that the exceptional locus of $\psi$ is the union of the $2d+1$ pairwise disjoint smooth curves $\Gamma, \ell_1,\dots,\ell_d,\ell'_1,\dots,\ell'_d$. 
For any such curve $C$, we denote by  $R_C$ the reduced part $\pi^{-1}(C)_{\red}$ of $\pi^{-1}(C)$. 
Let $E^X_i:=\pi^*E_i$ for any $i=1, \cdots, d$. 

Note that, locally around any point $s\in S$ which is not contained in $\Supp\{\psi^*A\}$, the morphism $\pi$ is a $\mathbb P^1$-bundle. 
Thus, we now study the morphism $\pi$ around a point $s\in \Supp\{\psi^*A\}$.

\medskip

\begin{nothing}[Zariski local description]\label{n-zariski-local}
Fix a closed point $s \in S$ such that $s\in C$ for some curve $C\subset \Ex(\psi)$. Then 
there exists an affine open neighbourhood $S^0=\Spec\,R$ of $s \in S$ such that $C=\Spec\,R/(f)$ 
for some $f \in R$ and 
\[
\begin{aligned}
\mathcal A_0|_{S^0}&=\MO_{S^0}\simeq R\\
\mathcal A_1|_{S^0}&=\MO_{S^0} \oplus \MO_{S^0}(\frac{1}{m_C}C)\simeq R^{\oplus 2}\\
&\cdots\\
\mathcal A_{m_C-1}|_{S^0}&=\MO_{S^0} \oplus \MO_{S^0}(\frac{1}{m_C}C) \oplus \cdots \oplus \MO_{S^0}(\frac{m_C-1}{m_C}C)\simeq R^{\oplus m_C}\\
\mathcal A_{m_C}|_{S^0}&=\MO_{S^0} \oplus \MO_{S^0}(\frac{1}{m_C}C) \oplus \cdots \oplus \MO_{S^0}(C)\simeq  R^{\oplus m_C} \oplus R[\frac{1}{f}]\\
&\cdots\\
\end{aligned}
\]
It follows that if $X^0:=\pi^{-1}(S^0)$, then 
\[
\mathcal A|_{S^0}=R[x, y, z]/(y^{m_C}-fz) \quad\text{and}\quad X^0=\Proj_R\, R[x, y, z]/(y^{m_C}-fz)
\]
where   
$x, y, z$ are homogeneous elements with $\deg x=\deg y=1$ and $\deg z=m_C$. 
Thus, $X^0$ is covered by the following three affine open subsets: 
\[
\begin{aligned}
D_+(x)&=\Spec\,R[y/x, z/x^{m_C}]/((y/x)^{m_C}-f(z/x^{m_C}))\\
&\simeq \Spec\,R[Y, Z]/(Y^{m_C}-fZ).\\
D_+(y)&=\Spec\,R[x/y, z/y^{m_C}]/(1-f(z/y^{m_C}))\simeq \Spec\,R[X, Z]/(1-fZ).\\
D_+(z)&=\Spec\,R[x^iy^j/z]_{i+j={m_C}}/(y^{m_C}/z-f).
\end{aligned}
\]
\end{nothing}

\begin{nothing}[Formally local description]\label{n-formal-local}
Fix a closed point $s \in S$ such that $s\in C$ for some curve $C\subset \Ex(\psi)$ and let $\widehat{R}$ be the formal completion of the local ring at $s$.  Let $R':=k[t_1, t_2]$ and 
\[
X^1:=\Proj\, R'[x, y, z]/(y^{m_C}-t_1z) \to \Spec R'.
\]
Then $X^0 \times_R \widehat{R} \simeq X^1 \times_{R'} \widehat{R'}$, 
where $\widehat{R'}$ is the formal completion at the origin. 
Thus, $X^1$ is covered by the following three open subsets: 
\[
\begin{aligned}
D_+(x) & \simeq \Spec\,R'[Y, Z]/(Y^{m_C}-t_1Z).\\
D_+(y)& \simeq \Spec\,R'[X, Z]/(1-t_1Z).\\
D_+(z)&=\Spec\,R'[x^iy^j/z]_{i+j={m_C}}/(y^{m_C}/z-t_1)\\
&\simeq \Spec\,k[t_2][x^iy^j/z]_{i+j={m_C}} 
\simeq \mathbb A^1_k \times_k D'
\end{aligned}
\]
where $D'=D_+(z') \subset \Proj \,k[x', y', z']=\mathbb P(1, 1, {m_C})$. 
\end{nothing}

\begin{lem}\label{l-X-CM}
$X$ is Cohen-Macaulay. 
\end{lem}

\begin{proof}
By the  faithfully flat descent property of being Cohen-Macaulay
\cite[Corollaly of Theorem 23.3]{matsumura89}, it is enough to show that   $X^1$, as defined  in \eqref{n-formal-local}, is Cohen-Macaulay. Thus, the claim follows easily. 
\end{proof}

\begin{lem}\label{l-equidim}
Any fibre of $\pi$ is irreducible and one-dimensional. 
\end{lem}

\begin{proof}
We first show that $\pi_*\MO_X=\MO_S$. 
Let 
\[
\pi\colon X \to S' \to S
\]
be the Stein factorisation of $\pi$. 
Note that $S'$ is a projective normal variety. 
Since there is a non-empty open subset of $S$ over which $\pi$ is a $\mathbb P^1$-bundle, 
it follows that $S' \to S$ is birational. 
Since $S$ is normal, Zariski main theorem implies that $S' \to S$ is an isomorphism. 
Thus, $\pi_*\MO_X=\MO_S$.

Let $s \in S$ be a closed point and let $K:=k(s)$  be the residue field at $s \in S$. 
If $s \not\in \Supp\{\psi^*A\}$, then the fibre $X_s$ is isomorphic to $\mathbb P^1_K$. 
Thus, we may assume that $s \in \Supp\{\psi^*A\}$ and in particular $s \in C$ for some curve $C\subset \Ex(\psi)$. We use the same notation as in (\ref{n-zariski-local}): 
\[
s\in S^0=\Spec\,R,\quad X^0=\Proj_R\, R[x, y, z]/(y^{m_C}-fz).
\]
Since $s \in C$, the image of $f$ in $K$ is zero. 
Thus, we have 
\[
X_K := \Proj_R\,R[x, y, z]/(y^{m_C}-fz) \times_R K \simeq \Proj\, K[x, y, z]/(y^{m_C}).
\]
Note that $D_+(y)$ is empty. Since $X_K$ is connected, it is enough to show that  
both  $D_+(x)$ and $D_+(z)$ are irreducible and one-dimensional. 

Since 
\[
D_+(x) \simeq \Spec\, K[Y, Z]/(Y^{m_C}),
\]
it follows that $D_+(x)$ is irreducible and one-dimensional. 
We have 
\[
D_+(z) \simeq \Spec\, K[x^iy^j/z]_{i+j={m_C}}/(y^{m_C}/z).
\]
Then 
\[
x^iy^j/z \in \sqrt{(y^{m_C}/z)}
\]
if $i+j={m_C}$ and $j \neq 0$. 
Thus, the $K$-algebra homomorphism 
\[
\theta\colon K[X] \to K[x^iy^j/z]_{i+j=m_C}/\sqrt{(y^{m_C}/z)}, \quad X \mapsto x^{m_C}/z
\]
is surjective. 
If $\Ker\,\theta \neq 0$, then $\dim K[X]/\Ker\,\theta=0$, 
which is a contradiction. 
Therefore, $\theta$ is an isomorphism and, in particular,  $D_+(z)$ is irreducible and one-dimensional, as claimed. 
\end{proof}

\begin{cor}\label{c-flat}
$\pi\colon X \to S$ is flat.
\end{cor}

\begin{proof}
By Lemma~\ref{l-X-CM},  $X$ is Cohen-Macaulay. Since $S$ is smooth, 
the claim follows from Lemma~\ref{l-equidim} and \cite[Theorem 23.1]{matsumura89}. 
\end{proof}

\begin{nothing}[Negative section]\label{n-negative}
We consider the graded $\MO_S$-algebra 
\[
\mathcal B^-=\MO_S \oplus \MO_S \oplus \cdots=\MO_S[t].
\]
There is a natural surjection 
\[
\mathcal A \to \mathcal B^-
\]
such that, if $s\in S$ is a closed point such that $s\in C$ for some curve $C\subset \Ex(\psi)$ and $S^0=\Spec R$ is an affine open neighbourhood of $s$, then using the same notation as in 
\eqref{n-zariski-local}, we have 
\[
\mathcal A|_{S^0}=R[x, y, z]/(y^{m_C}-fz) \to R[t]=\mathcal B^-|_{S^0}.
\]
where  $x \mapsto t$, $y\mapsto 0$, $z \mapsto 0$. 

This induces a closed immersion over $S$
\[
S^-:=\Proj_S\, \mathcal B^- \hookrightarrow \Proj_S \,\mathcal A=X
\]
such that  $S^-$ is a section of $\pi\colon \Proj_S \,\mathcal A=X \to S$. 

In particular, if we denote $X^0=\Proj_R\,R[x, y, z]/(y^{m_C}-fz)$ as in \eqref{n-zariski-local}, 
then $S^-\cap X^0$ is given by $\{y=z=0\}\subset X^0$. 
\end{nothing}

\begin{nothing}[Positive section]\label{n-positive}
We now consider the graded $\MO_S$-algebra: 
\[
\mathcal B^+=\MO_S \oplus \MO_S(\psi^*A) \oplus \MO_S(2\psi^*A) \oplus 
\MO_S(3\psi^*A)\oplus \cdots.
\]

There is a natural surjection 
\[
\mathcal A \to \mathcal B^+
\]
such that, if $s\in S$ is a closed point such that $s\in C$ for some curve $C\subset \Ex(\psi)$ and $S_0=\Spec R$ is an affine open neighbourhood of $s$, then using the same notation as in 
\eqref{n-zariski-local}, we have 
\[
\mathcal A|_{S^0}=R[x, y, z]/(y^{m_C}-fz) \to R[u, v]/(u^{m_C}-fv)=\mathcal B^+|_{S^0}.
\]
where  $x \mapsto 0$, $y\mapsto u$, $z \mapsto v$. 

This induces a closed immersion over $S$
\[
S^+:=\Proj_S \,\mathcal B^+ \hookrightarrow \Proj_S \,\mathcal A=X.
\]
In particular, if we denote $X^0=\Proj_R\, R[x, y, z]/(y^{m_C}-fz)$ as in \eqref{n-zariski-local}, 
then  $S^+\cap X^0$ is given by $\{x=0\}\subset X^0$.

\begin{lem}$S^+$ is a section of $\pi\colon X \to S$. 
\end{lem}

\begin{proof}
It is enough to prove the result locally. Thus, it suffices to show that $S^+\cap X^0$ is a section of $\pi|_{X^0}\colon X^0\to S^0$. 
We may write  
\[
S^+ \cap X^0=\Proj_R\, (R[y, z]/(y^{m_C}-fz)) \subset \mathbb P_R(1, m_C) \simeq \mathbb P^1_R
\]
which is covered by the affine open sets $D_+(y)$ and $D_+(z)$. 
Clearly, $\{z=0\} \cap (S^+\cap X^0)=\emptyset$.

Therefore, we get natural isomorphisms: 
\[
\begin{aligned}
S^+|_{X^0} & =D_+(y) \cup D_+(z)=D_+(z)\\
& \simeq \Spec\,R[y^{m_C}/z]/(y^{m_C}/z-f)\\
&=\Spec\,R[Y]/(Y-f) \simeq \Spec\,R=S^0.
\end{aligned}
\]
Thus,  the natural morphism $S^+ \to S$ is an isomorphism. 
\end{proof}
\end{nothing}

\subsection{Resolution}\label{ss-resolution}

The purpose of this subsection is to describe an explicit resolution of singularities for $X$.

By the formally local description (\ref{n-formal-local}), we see that the singular locus $\Sing\, X$ of $X$ 
can be written as   
\[
\Sing\, X=\bigcup_{C \subset \Supp \{\psi^*A\}}(C^+ \cup C^-),
\]
where, for any curve $C\subset \Supp \{\psi^*A\}$,  we denote by $C^{\pm}$  the integral scheme which is set-theoretically equal to $R_C \cap S^{\pm}$. 
 By (\ref{n-formal-local}), we also have:
\begin{itemize}
\item Up to formal completion, locally around any point  $s\in C^+$, $X$ is isomorphic to $\mathbb A_k^1\times_k D'$, 
where $D'=D_+(z') \subset \Proj\,k[x', y', z']=\mathbb P(1, 1, m_C)$. 
\item Up to formal completion, locally around any point  $s\in C^-$, $X$ is isomorphic  to 
the direct product of $\mathbb A_k^1$ and a surface with a canonical $A_{m_C-1}$-singularity. 
\end{itemize}

Thus, we can construct a resolution of $X$ by considering the minimal resolution of a surface singularity. 
More precisely, we construct a resolution 
\[
\mu\colon \widetilde X \to X
\]
as follows.
For any $C \subset \Supp\{\psi^*A\}$, we first take the blow-up of $X$ along $C^+$. 
Then we take the blow-up of $X$ along $C^-$, $\ulcorner\frac{m_C-1}{2}\urcorner$ times. 
For any prime divisor $D$ on $X$, we denote by $\widetilde D$, its proper transform on $\widetilde X$. 
Let $F_C^+$ be the unique $\mu$-exceptional prime divisor over $C^+$ and 
let $F^-_1, F^-_2, \cdots, F^-_{m_C-1}$ be the   $\mu$-exceptional prime divisors over $C^-$
whose (extended) dual graph is given by: 
\[
\widetilde S^--F^-_1-F^-_2-\cdots-F^-_{m_C-1}-\widetilde R_C.
\]

\begin{lem}\label{l-X-rational}
The following  hold:
\begin{enumerate}
\item  $H^i(X, \MO_X)=0$  for any $i>0$. 
\item 
Given any proper birational morphism $h\colon W \to X$ from a smooth proper threefold $W$, we have $R^ih_*\MO_W=0$  for any $i>0$.
\end{enumerate}
\end{lem}
 
\begin{proof} 
By the description above, we have that $R^i\mu_*\MO_{\widetilde X}=0$ for any $i>0$. 
Thus, (2) follows from \cite[Theorem~1]{CR11}.
Since $X$ is a rational variety, (2) implies  (1).
\end{proof}

\begin{lem}\label{l-X-Qfac}
$X$ is $\Q$-factorial. 
\end{lem}
 
\begin{proof}
Fix a closed point $x \in X$. 
By \cite[(24. E)]{matsumura80}, it is enough to show that $\Spec\,\widehat{\MO}_{X, x}$ is $\Q$-factorial. 
By \eqref{n-formal-local}, $\Spec\,\widehat{\MO}_{X, x}$ is the completion of 
the direct product of $\mathbb A^1_k$ and a surface klt singularity. 
Thus, we can apply the same proof as in the case of a  surface  
(e.g. see  the proof of \cite[Theorem~5.3]{tanaka12}). 
\end{proof}

\begin{lem}\label{l-pic-X}
The following  hold:
\begin{enumerate}
\item 
Let $\zeta$ be a fibre of $\pi$. 
Then the  sequence 
\[
0 \to \Pic (S)_{\Q} \xrightarrow{\pi^*} \Pic (X)_{\Q} \xrightarrow{\cdot \zeta} \Q \to 0
\]
is exact. 
\item 
$\rho(X)=\rho(S)+1$.
\end{enumerate}
\end{lem}

\begin{proof}
We first show (1). 
Let $D$ be a Cartier divisor on $X$ such that $D \cdot \zeta=0$. 
It is enough to show that $D=\pi^*D_S$ for some $\Q$-divisor ${D_S}$ on $S$. 
This follows from Corollary~\ref{c-flat} and 
the fact that the generic fibre of $\pi$ is $\mathbb P^1_{K(S)}$.  
Thus, (1) holds and  (2) follows immediately. 
\end{proof}

\begin{lem}\label{l-resolution}
With the same notation as above, the following  hold:
\begin{enumerate}
\item Each $\mu$-exceptional prime divisors $F$ and $\widetilde R_C$ 
is a Hirzebruch surface, 
i.e. a $\mathbb P^1$-bundle over $\mathbb P^1$. 
\item 
The birational morphism $\mu$ is a log resolution of 
\[
\left(X, S^++S^-{+\sum_{C \subset \Ex(\psi)} R_C}\right).
\] 
\item $K_{\widetilde X}+\sum_{C \subset \Supp \{\psi^*A\}} \frac{m_C-2}{m_C}F_C^+=\mu^*K_X$. 
\item $\mu^*S^+=\widetilde S^++\sum_{C \subset \Supp \{\psi^*A\}}\frac{1}{m_C}F^+_C$. 
\item $\mu^*S^-=\widetilde S^-+\sum_{C \subset \Supp \{\psi^*A\}}(\frac{m_C-1}{m_C}F^-_1+\frac{m_C-2}{m_C}F^-_2+\cdots+\frac{1}{m_C}F^-_{m_C-1}).$
\item For any $C \subset \Supp \{\psi^*A\}$, 
\[
\mu^*R_C=\widetilde R_C+\frac{1}{m_C}F_C^++\frac{1}{m_C}F^-_1+\frac{2}{m_C}F^-_2+\cdots+\frac{m_C-1}{m_C}F^-_{m_C-1}.
\]
\end{enumerate}
\end{lem}

\begin{proof}
Note that, to prove the Lemma, we will use 
the base change defined by $\Spec\,\widehat \MO_{S, s} \to S$ as in \eqref{n-formal-local}. 

We first show (1). 
We have that $\mu|_{\widetilde R_C}\colon \widetilde R_C \to R_C$ is an isomorphism 
because it is an isomorphism after taking the faithfully flat base change $(-)\otimes_R \widehat R$. 
Each $\mu$-exceptional prime divisor $F$ is a $\mathbb P^1$-bundle over $C^+$ or $C^-$ 
because it is so after considering the base change $(-)\otimes_S \widehat \MO_{S, s}$. 
Thus, (1) holds. 

We now show (2). Let
\[
K:=\Ex(\mu) \cup \widetilde S^+ \cup \widetilde S^- \cup \left(\bigcup_{C\subset \Ex(\psi)} \widetilde{R}_C\right).
\]
If $D_1, D_2, D_3$ are distinct prime divisors contained in 
$K$, then $D_1 \cap D_2 \cap D_3=\emptyset$. 
Thus, it suffices to show that 
if two prime divisors $D_1$ and $D_2$, contained in  $K$,
intersect, then the scheme-theoretic intersection  $D_1 \cap D_2$ is smooth. 
The natural morphism from $D_1 \cap D_2$ onto either $C^+$ or $C^-$ 
is an isomorphism since it is so 
after taking the base change $(-)\times_S \widehat \MO_{S, s}$. 
Thus, (2) holds. 

Similarly,  (3)--(6) follow
after taking the base change $\Spec\,\widehat \MO_{S, s} \to S$. 
\end{proof}

\begin{lem}\label{l-different}
Let $m$ be a sufficiently divisible positive integer. 
After identifying $S$ with $S^+$ (resp. $S^-$),  
the following isomorphisms hold: 
\[
\MO_X(m(K_X+S^+))|_{S^+} \simeq \MO_{S}(m(K_{S}+\sum_{C \subset \Ex(\psi)} \frac{m_C-1}{m_C}C)),
\]
\[
\MO_X(m(K_X+S^-))|_{S^-} \simeq \MO_{S}(m(K_{S}+\sum_{C \subset \Ex(\psi)} \frac{m_C-1}{m_C}C)).
\]
\end{lem}

\begin{proof}
The claim follows from Lemma~\ref{l-resolution}.
\end{proof}

\begin{cor}\label{c-different}
Let $m$ be a sufficiently divisible positive integer. 
Then the following holds: 
\[
m(K_X+S^++S^-) \sim \pi^*(m(K_S+\sum_{C \subset \Ex(\psi)} \frac{m_C-1}{m_C}C)).
\]
\end{cor}

\begin{proof}
The claim follows from Lemma~\ref{l-different}. 
\end{proof}

\subsection{The inverse image of the bad locus}

Recall that,  for any curve $C \subset \Ex(\psi)$,  we denote by $R_C$  the reduced part   $\pi^{-1}(C)_{\red}$ of $\pi^{-1}(C)$.

\begin{lem}\label{l-RC}
For any curve $C \subset \Ex(\psi)$, the following  hold
\begin{enumerate}
\item $\pi^*C=m_CR_C$.  
\item $R_C \simeq \mathbb P^1 \times \mathbb P^1$. 
Furthermore, the induced morphism $R_C \to C$ is a $\mathbb P^1$-bundle. 
\item $(C^+\text{ in }R_C)^2=(C^-\text{ in }R_C)^2=0$. 
\item $(m_CS^{+})|_{R_C} = C^+, \quad (m_CS^-)|_{R_C} = C^-.$ 
\item $S^+ \cdot C^+=S^- \cdot C^-=0$. 
\item $K_X \cdot C^+= K_X \cdot C^-=\frac{-C^2-2m_C}{m_C}$. 
\end{enumerate}
\end{lem}

\begin{proof}
By \cite[Proposition 3.5.3]{egaii}, 
the scheme-theoretic inverse image $\pi^{-1}(C)$ can be written as 
$\pi^{-1}(C)=\Proj_C\, \mathcal A|_C.$ 
Let $s \in C$ be a closed point and let $X^0$ be as in  (\ref{n-zariski-local}). 
We obtain  
\[
\pi^{-1}(C)|_{X^0}=\Proj\, R[x, y, z]/(y^{m_C}-fz, f) \simeq \Proj\, (R/f)[x, y, z]/(y^{m_C}).
\]
Since $R_C=\pi^{-1}(C)_{\red}$,  \cite[Proposition 3.1.13]{egaii} imply that 
\[
R_C=\Proj_C\, (\mathcal A|_C)_{\red},
\]
and
\[R_C|_{X^0}=\Proj\, (R/f)[x, y, z]/(y) \simeq \Proj (R/f)[x, z].
\]
Thus, (1) holds. 
Since $\MO_S(m_C\psi^*A)|_C \simeq \MO_C$, we can  check that 
\[
R_C=\Proj_C\, (\mathcal A|_C)_{\red}\simeq C \times_k \mathbb P_k(1, m_C).
\]
Since $C \simeq \mathbb P^1_k$ and $\mathbb P_k(1, m_C) \simeq \mathbb P^1_k$,  (2) holds.

We now show (3). 
Since $C^+ \cap C^-=\emptyset$, each divisor $C^{\pm}$ on $R_C$ is not ample. 
Thus,  (3) follows from the fact that 
any curve $B$ on $\mathbb P^1 \times \mathbb P^1$, which is not ample, satisfies $B^2=0$. 

We now show (4). By (4) of Lemma \ref{l-resolution}, it follows that $m_CS^+$ is Cartier. Thus, 
we may write $(m_CS^{+})|_{R_C} = xC^{+}$ for some positive integer $x$. 
Since
\[
\MO_X(m_CR_C)|_{R_C} \simeq \pi^*\MO_S(C)|_{R_C} 
\simeq (\pi|_{R_C})^*\MO_C(C^2)
\]
and
\[
m_CR_C \cdot m_CR_C \cdot S^{\pm}=\pi^*C \cdot \pi^*C \cdot S^{\pm}=C^2,
\]
we have that 
\[
\begin{aligned}
C^2&=m_CR_{C} \cdot m_CR_C \cdot S^+=(m_CS^+)|_{R_C} \cdot (m_CR_C)|_{R_C}\\
&=x C^+ \cdot m_CR_{C}=x C^+ \cdot \pi^*C=x C^2.
\end{aligned}
\]
Thus, $x=1$.  Similarly we obtain $(m_CS^{-})|_{R_C}=C^{-}$ and (4) holds. 

(3) implies that 
\[
m_CS^{\pm} \cdot C^{\pm}=m_CS^{\pm} \cdot m_CS^{\pm} \cdot R_C=(C^{\pm}\text{ in }R_C)^2=0.
\]
Thus, (5) holds. 

We finally show (6). 
Around $R_C$, we have 
\[
m_C(K_X+S^++S^-) \sim_{\Q} \pi^*(m_C(K_S+\frac{m_C-1}{m_C}C)).
\]
Moreover,  
\[
(m_C(K_S+\frac{m_C-1}{m_C}C)) \cdot C=m_C(K_S+C) \cdot C-C^2=-C^2-2m_C.
\]
Thus,  (4) implies 
\[
m_CK_X|_{R_C}+C^++C^- \sim_{\Q} (\pi|_{R_C})^*M
\]
for some Cartier divisor $M$ on $C$ of degree $-C^2-2m_C$. 
By (3), we have
\[\begin{aligned}
m_CK_X \cdot C^{\pm} &=
(m_CK_X|_{R_C}+C^++C^-)\cdot C^{\pm} \\
&= (\pi|_{R_C})^*M \cdot C^{\pm}=-C^2-2m_C.
\end{aligned}\]
Thus, (6) holds. 
\end{proof}

\begin{lem}\label{l-normal-bdl}
Let $m$ be a sufficiently divisible positive integer. After identifying $S$ with $S^+$ (resp. $S^-$), we have
\begin{enumerate}
\item $\MO_X(mS^+)|_{S^+} \simeq \MO_S(m\psi^*A),\quad \MO_X(mS^-)|_{S^-} \simeq \MO_S(-m\psi^*A)$. 
\item $m(S^+-S^-) \sim \pi^*(m\psi^*A).$
\end{enumerate}
\end{lem}

\begin{proof}
Since, outside $\Supp\{\psi^*A\}$,  the morphism $\pi\colon X \to S$ coincides with the $\mathbb P^1$-bundle 
\[
\mathbb P_S(\MO_S \oplus \MO_S(\llcorner \psi^*A\lrcorner)) \to S
\]
 we have that 
\[
\MO_X(mS^{\pm})|_{S^{\pm}} \simeq \MO_S(\pm m\psi^*A+\sum_{C \subset \Ex(\psi)} x_{C^\pm }C)
\]
for some integers $x_{C^\pm}$. 
We have that   $\psi^*A \cdot C=0$ and, by  Lemma~\ref{l-RC}, 
$S^{\pm} \cdot C^{\pm}=0$. In particular, $x_{C^\pm}=0$ for any curve $C\subset \Ex(\psi)$. Thus, (1) holds and (2) follows immediately from (1). 
\end{proof}

\begin{lem}\label{l-X-plt}
The following  hold:
\begin{enumerate}
\item For any $0 \leq \alpha <1$, the pair 
\[
(X, S^++S^-+\alpha(\sum_{C \subset \Ex(\psi)}R_C+\sum_{i=1}^dE^X_i))
\]
is plt. 
\item $(X, S^++S^-+\sum_{C \subset \Ex(\psi)}R_C+\sum_{i=1}^dE^X_i)$ is log canonical. 
\item For any $0 \leq \alpha <1$, the pair 
\[
(X, \alpha(S^++S^-+\sum_{C \subset \Ex(\psi)}R_C)+\sum_{i=1}^dE^X_i)
\]
is plt. 
\end{enumerate}
\end{lem}

\begin{proof}
We first show (1). After identifying  $S$ with $S^{\pm}$,
Lemma~\ref{l-different}  implies
\[
(K_X+S^{\pm})|_{S^{\pm}}=K_{S}+\sum_{C \subset \Ex(\psi)} \frac{m_C-1}{m_C}C
\]
and (1) of Lemma~\ref{l-RC} implies
\[
R_C|_{S^{\pm}}=\frac{1}{m_C}C.
\]
Thus, 
\[\begin{aligned}
\left(K_X+S^++S^-+\alpha(\sum_{C \subset \Ex(\psi)}R_C+\sum_{i=1}^dE^X_i)\right)\Big{|}_{S^{\pm}}\\
=K_S+\sum_{C \subset \Ex(\psi)}\frac{m_C-1+\alpha}{m_C}C+\alpha \sum_{i=1}^dE_i,
\end{aligned}\]
and since 
$(S, \sum_{C \subset \Ex(\psi)}\frac{m_C-1+\alpha}{m_C}C+\alpha \sum_{i=1}^dE^X_i)$ 
is strongly $F$-regular, 
\cite[Theorem A]{Das15} implies that 
the pair 
\[
(X, S^++S^-+\alpha( \sum_{C \subset \Ex(\psi)}R_C+\sum_{i=1}^dE^X_i))
\]
 is plt around $S^+$ and $S^-$. 
 We now show that it is klt outside $S^+ \cup S^-$.
Let $X':=X \setminus (S^+ \cup S^-)$. 
Since $\pi:X \to S$ is a $\mathbb P^1$-bundle outside $\bigcup_C R_C$, 
it is enough to show that, for any curve $C \subset \Ex(\psi)$,  the pair
\[
(X', (R_C+\alpha \sum_{i=1}^dE^X_i)|_{X'})
\]
 is plt around $R_C$. 
Since  $R_C \to C$ is a $\mathbb P^1$-bundle ((2) of Lemma \ref{l-RC}), 
$(R_C \cap X', \alpha \sum_{i=1}^d(E^X_i)|_{R_C \cap X'})$ is strongly $F$-regular. 
Therefore, \cite[Theorem A]{Das15} implies that $(X', (R_C+ \alpha \sum_{i=1}^dE^X_i)|_{X'})$ is plt around $R_C$, as desired. 
Thus, (1) holds and (2) follows immediately from (1).

We now show (3). 
Let 
\[
\Delta:=\alpha(S^++S^-+\sum_{C \subset \Ex(\psi)}R_C)+\sum_{i=1}^dE^X_i.
\]
By contradiction, we assume that there exists 
a proper birational morphism $\theta\colon W \to X$ and 
a $\theta$-exceptional prime divisor $D$ on $W$ with 
log discrepancy $a(D, X, \Delta)=0$. 
By (1), we have that the centre $\theta(D)$ of $D$ in $X$ is contained in $E^X_i$ for some $i=1,\dots,d$ 
and, by (2), we have that  $\theta(D)$ is not contained in 
any irreducible component of $\Supp\{\Delta\}$. 
Fix a general point $x \in \theta(D)$. 
In particular, $x \not\in \Supp (\Delta -E^X_i)$. 
Locally around $x$, the morphism $\pi$ is a $\mathbb P^1$-bundle over $S$. 
Therefore, $(X, \Delta)$ is plt around $x$, a contradiction.  
Thus, (3) holds. 
\end{proof}

\subsection{Contraction of the negative section}

\subsubsection{Construction of $f\colon X \to Y$}
Recall that we have defined the morphisms
\[ 
X \xrightarrow{\pi} S \xrightarrow{\psi} T
\]
and for any curve $C \subset \Ex(\psi)$, Lemma \ref{l-RC} implies that 
\[
\pi|_{R_C}:R_C\simeq \mathbb P^1 \times \mathbb P^1 \to C\simeq  \mathbb P^1.
\] 
Given an ample $\Q$-divisor $H$ on $T$, we define
\[
L=S^++\pi^*\psi^*H.
\]
Then $L$ is big. By (2) of Lemma~\ref{l-normal-bdl}, it follows that   $S^+$ is semi-ample and thus, so is $L$. 
Let 
\[
f\colon X \to Y
\]
be the  birational morphism induced by $L$ so that $f_*\MO_X=\MO_Y$. 
It follows that for a curve $B$ in $X$, 
the image $f(B)$ is a point if and only if $B$ is contained in $R_C$ for some 
$C \subset \Ex(\psi)$ and $B$ is a fibre of 
the projection 
$R_C=\mathbb P^1 \times \mathbb P^1 \to \mathbb P^1$ other than $\pi|_{R_C}$. 
In particular, we get a commutative diagram
\[
\begin{CD}
X @>f>> Y\\
@VV\pi V @VV\pi_YV\\
S @>\psi>> T.
\end{CD}
\]
We define $T^+:=f_*S^+$, $T^-:=f_*S^-$ and $E_i^Y:=f_*(E^X_i).$ 
For each $i=1,\dots,d$, let $E_i^+$ and $E_i^-$ be the curves on $S^+$ and $S^-$ corresponding to $E_i$ on $S$.

\begin{lem}\label{l-Y-Qfac}
With the same notation as above, the following  hold: 
\begin{enumerate}
\item $Y$ is $\Q$-factorial. 
\item $\rho(Y)=2$. 
\end{enumerate}
\end{lem}

\begin{proof}
Fix a curve $C \subset \Ex(\psi)$. 
Given  an ample divisor $N$ on $X$, 
we define $M:=N+\lambda R_C$, where 
\[
\lambda:=\max\{\lambda \in \R_{\geq 0}\,|\, N+\lambda R_C\text{ is nef}\}.
\]
By (2) of Lemma \ref{l-RC}, it follows that any nef divisor on $R_C$ is semi-ample. Thus,   \cite[Proposition 1.6]{keel99} implies that $M$ is semi-ample and  it induces birational morphism 
$f'\colon X \to Y'$
such that $f'_*\MO_X=\MO_{Y'}$, 
$\Ex(f')=R_C$ and $f'\colon R_C \to f'(R_C)$ is the projection 
other than $\pi|_{R_C}\colon R_C \to C$. 
Then Lemma~\ref{l-Qfac-criterion} implies that $Y'$ is $\mathbb Q$-factorial. For any curve $C\subset \Ex(\psi)$, there exists an open neighbourhood of $f(R_C)$ in $Y$, which  is isomorphic to a Zariski open subset of $Y'$. Thus,  (1) holds and 
 Lemma~\ref{l-pic-X} implies (2). 
\end{proof}

\begin{lem}\label{l-Y-discrep}
The following  hold:
\begin{enumerate}
\item 
$K_X+\sum_{C \subset \Ex(\psi)}\frac{-C^2-2m_C}{-C^2}R_{C}=f^*K_Y.$ 
\item For any $1\leq i \leq d$ and $0\leq \beta<1$, 
the pair $(Y, E^Y_i+\beta T^++\beta T^-)$ is plt.  
\end{enumerate}
\end{lem}

\begin{proof}
By (1) and (6) of Lemma~\ref{l-RC}, we have that 
$K_X\cdot C^+=\frac{-C^2-2m_C}{m_C}$ 
and $m_CR_C \cdot C^+=\pi^*C \cdot C^+=C^2$. Thus,
(1) holds.

We now show (2). 
Since 
\[
\begin{aligned}
f^*E^Y_i=f^*\pi_Y^*E^T_i=\pi^*(E_i+\frac{1}{2}\ell_i+\frac{1}{2}\ell'_i+\frac{1}{2d-4}\Gamma)\\
=E_i^X+\frac{m_{\ell_i}}{2}R_{\ell_i}+\frac{m_{\ell'_i}}{2}R_{\ell_i'}+\frac{m_{\Gamma}}{2d-4}R_{\Gamma},
\end{aligned}
\]
we have that 
\[
\begin{aligned}
f^*(K_Y+E^Y_i+\beta T^++\beta T^-)
=K_X+\sum_{C \subset \Ex(\psi)}\frac{-C^2-2m_C}{-C^2}R_{C}\\
+(E^X_i+\frac{m_{\ell_i}}{2}R_{\ell_i}+\frac{m_{\ell'_i}}{2}R_{\ell_i'}+\frac{m_{\Gamma}}{2d-4}R_{\Gamma})+\beta S^++\beta S^-\\
\leq K_X+E^X_i+(1-\epsilon )(S^++S^-+\sum_{C \subset \Ex(\psi)} R_C)
\end{aligned}
\]
for some $0<\epsilon<1$. 
Thus, (3) of Lemma~\ref{l-X-plt} implies that $(Y, E^Y_i+\beta T^++\beta T^-)$ is plt and (2) holds. 
\end{proof}

\begin{lem}\label{l-XY}
Fix $1 \leq i \leq d$.  
Then the following  hold:
\begin{enumerate}
\item $(\sum_{C \subset \Ex(\psi)}\frac{-C^2-2m_C}{-C^2}R_{C}) \cdot E_i^{\pm}=\frac{-\Gamma^2-2m_{\Gamma}}{-\Gamma^2 \cdot m_{\Gamma}}+
\frac{1-m_{\ell_i}}{m_{\ell_i}}+\frac{1-m_{\ell'_i}}{m_{\ell'_i}}.$
\item $S^+ \cdot E^+_j=\psi^*A \cdot E_j, \quad S^- \cdot E^-_j=-\psi^*A \cdot E_j$ for any $j$. 
\item 
\[
\begin{aligned}
K_X\cdot E^+_i=\frac{m_{\Gamma}-1}{m_{\Gamma}}+\frac{m_{\ell_i}-1}{m_{\ell_i}}+\frac{m_{\ell'_i}-1}{m_{\ell'_i}}-1-\psi^*A \cdot E_i,\\
K_X\cdot E^-_i=\frac{m_{\Gamma}-1}{m_{\Gamma}}+\frac{m_{\ell_i}-1}{m_{\ell_i}}+\frac{m_{\ell'_i}-1}{m_{\ell'_i}}-1+\psi^*A \cdot E_i.
\end{aligned}
\]
\item 
\[
\begin{aligned}
K_Y \cdot f(E^+_i)=\frac{m_{\Gamma}-1}{m_{\Gamma}}-1-\psi^*A \cdot E_i+\frac{-\Gamma^2-2m_{\Gamma}}{-\Gamma^2 \cdot m_{\Gamma}},\\
K_Y \cdot f(E^{-}_i)=\frac{m_{\Gamma}-1}{m_{\Gamma}}-1+\psi^*A \cdot E_i+\frac{-\Gamma^2-2m_{\Gamma}}{-\Gamma^2 \cdot m_{\Gamma}}.
\end{aligned}
\]
\item $E_i^Y \cdot f(E^{+}_j)=E_i^Y \cdot f(E^-_j)=\frac{1}{2d-4}$ for any $j=1,\dots,d$.  
\end{enumerate}
\end{lem}

\begin{proof}
By (1) of Lemma~\ref{l-RC}, we have 
\[
\begin{aligned}
(\sum_{C \subset \Ex(\psi)}\frac{-C^2-2m_C}{-C^2}R_{C}) \cdot E_i^{\pm}
&=(\sum_{C \subset \Ex(\psi)}\frac{-C^2-2m_C}{-C^2}\frac{\pi^*C}{m_C}) \cdot E_i^{\pm}\\
&=\sum_{C \subset \Ex(\psi)}\frac{-C^2-2m_C}{-C^2}\frac{C \cdot E_i}{m_C}\\
&=\frac{-\Gamma^2-2m_{\Gamma}}{-\Gamma^2 \cdot m_{\Gamma}}+
\frac{1-m_{\ell_i}}{m_{\ell_i}}+\frac{1-m_{\ell'_i}}{m_{\ell'_i}}.
\end{aligned}
\]
Thus, (1) holds. (1) of Lemma~\ref{l-normal-bdl} implies (2). 

Corollary~\ref{c-different} implies 
\[
\begin{aligned}
(K_X+S^++S^-) \cdot E^{\pm}_i
&=\pi^*(K_S+\sum_C \frac{m_C-1}{m_C}C) \cdot E^{\pm}_i
=(K_S+\sum_C \frac{m_C-1}{m_C}C) \cdot E_i\\
&=\frac{m_{\Gamma}-1}{m_{\Gamma}}+\frac{m_{\ell_i}-1}{m_{\ell_i}}+\frac{m_{\ell'_i}-1}{m_{\ell'_i}}-1,
\end{aligned}
\]
Thus, (2) implies (3). 
 
By (1) of Lemma~\ref{l-Y-discrep}, we have 
\[
K_Y \cdot f(E^{\pm}_i)=f^*K_Y \cdot E^{\pm}_i=
(K_X+\sum_{C \subset \Ex(\psi)}\frac{-C^2-2m_C}{-C^2}R_{C}) \cdot E^{\pm}_i\]
Thus, (1) and (3) imply (4).

We have
\[
\begin{aligned}
E_i^Y \cdot f(E^{\pm}_j)&=f^*E_i^Y \cdot E^{\pm}_j=f^*\pi_Y^*E^T_i \cdot E^{\pm}_j\\
&=\pi^*(E_i+\frac{1}{2}(\ell_i+\ell'_i)+\frac{1}{2d-4}\Gamma) \cdot E^{\pm}_j\\
&=(E_i+\frac{1}{2}(\ell_i+\ell'_i)+\frac{1}{2d-4}\Gamma) \cdot E_j=\frac{1}{2d-4}.
\end{aligned}
\]
Thus, (5) holds. 
\end{proof}

\subsubsection{Relative Kawamata--Viehweg vanishing theorem}

Recall that we have defined the  birational morphisms
\[
\widetilde f\colon \widetilde X \xrightarrow{\mu} X \xrightarrow{f} Y.
\]

\begin{lem}\label{l-tildeX-kvv}
Let $L$ be a Cartier divisor on $\widetilde X$. 
Then the following  hold:
\begin{enumerate}
\item 
Assume that $L-(K_{\widetilde X}+\Delta)$ is $\mu$-nef for some 
$\widetilde f$-exceptional effective $\Q$-divisor $\Delta$ 
such that $(X, \Delta)$ is klt. 
Then $R^i\mu_*\MO_{\widetilde X}(L)=0$ for any $i>0$. 
\item 
Assume that $L-(K_{\widetilde X}+\Delta)$ is $\widetilde f$-nef for some 
$\widetilde f$-exceptional effective $\Q$-divisor $\Delta$ 
such that $(X, \Delta)$ is klt. 
Then $R^i\widetilde f_*\MO_{\widetilde X}(L)=0$ for any $i>0$. 
\end{enumerate}
\end{lem}

\begin{proof}
By (2) of Lemma~\ref{l-resolution}, it follows that $\Supp\,\Ex(\widetilde{f})$ is 
a simple normal crossing divisor. 
As $Y$ is $\Q$-factorial (Lemma \ref{l-Y-Qfac}), 
there exists an $\widetilde{f}$-exceptional effective $\Q$-divisor $E$ such that 
$-E$ is $\widetilde{f}$-ample \cite[Lemma 2.62]{km98}. 
After possibly replacing $\Delta$ by $\Delta + \epsilon E$ for some small positive rational number $\epsilon$, 
we may assume that $L-(K_{\widetilde X}+\Delta)$ is $\mu$-ample in (1) and 
that $L-(K_{\widetilde X}+\Delta)$ is $\widetilde f$-ample in (2). 
Then both (1) and (2)  follow 
from Proposition~\ref{p-birat-kvv}. 
\end{proof}

\begin{lem}\label{l-f-kvv}
Let $D$ be a $\Q$-Cartier $\Z$-divisor on $X$. 
If $D-(K_X+\Delta)$ is $f$-nef 
for some $f$-exceptional effective $\Q$-divisor $\Delta$ 
such that $(X, \Delta)$ is klt, 
then $R^if_*\MO_X(D)=0$ for any $i>0$. 
\end{lem}

\begin{proof}
Let $M:=\ulcorner \mu^*(D)+K_{\widetilde X}-\mu^*(K_X+\Delta)\urcorner.$ 
Since $(X, \Delta)$ is klt, we have that $\mu_*\MO_{\widetilde X}(M)=\MO_X(D)$. 
 Lemma~\ref{l-tildeX-kvv} implies that for any $i>0$
\[
R^i\mu_*\MO_{\widetilde X}(M)=0 \quad \text{and}\quad
R^i\widetilde f_*\MO_{\widetilde X}(M)=0.
\]
Thus, the claim follows.
\end{proof}

\begin{lem}\label{l-f-log-fano}
There exists an $f$-exceptional effective $\Q$-divisor $\Delta$ on $X$ such that 
$(X, \Delta)$ is klt and $-(K_X+\Delta)$ is $f$-nef. 
\end{lem}

\begin{proof}
Lemma~\ref{l-X-plt} and Lemma~\ref{l-Y-discrep} imply the claim. 
\end{proof}

\begin{lem}\label{l-Y-rational}
The following  hold:
\begin{enumerate}
\item  $H^i(Y, \MO_Y)=0$  for any $i>0$. 
\item 
Given any proper birational morphism $h\colon W \to Y$ from a smooth proper threefold $W$, 
we have $R^ih_*\MO_W=0$  for any $i>0$.
\end{enumerate}
\end{lem}
 
\begin{proof}
Lemma~\ref{l-f-kvv} and Lemma~\ref{l-f-log-fano} imply that $R^if_*\MO_X=0$ for any $i>0$. 
Thus, both (1) and (2)  follow from Lemma~\ref{l-X-rational}. 
\end{proof}

\begin{lem}\label{l-T-normal}
Both  $T^+$ and $T^-$ are normal, and the induced morphisms 
$T^+ \to T$ and $T^- \to T$ are isomorphisms. 
\end{lem}

\begin{proof}
It is enough to show that $T^+$ and $T^-$ are normal. 
Since the proof is the same, we only show the normality of $T^+$. 
It is enough to prove that the induced homomorphism 
\[
\MO_{T^+} \to f_*\MO_{S^+}
\]
is surjective. 
We have the exact sequence: 
\[
0 \to \MO_X(-S^+) \to \MO_X \to \MO_{S^+} \to 0.
\]
By (5) of Lemma \ref{l-RC}, we have that $S^+ \cdot \zeta=0$ for any $f|_{S^+}$-exceptional curve $\zeta$. Thus, 
$S^+$ is $f$-numerically trivial, and Lemma~\ref{l-f-kvv} and Lemma~\ref{l-f-log-fano} imply that $R^1f_*\MO_X(-S^+)$, as claimed.  
\end{proof}

\subsubsection{Construction of $g\colon Y \to Z$}\label{ss-construction-Z}

By  Lemma \ref{l-T-normal}, we may identify $T^+$ (resp. $T^-$) with $T$, so that we get the $\Q$-linear equivalences 
\[
\MO_Y(T^+)|_{T^+} \sim_{\Q} A\qquad \MO_Y(T^-)|_{T^-} \sim_{\Q} -A
\]
and
\[
T^+-T^- \sim_{\Q} \pi_Y^*A.\]
Thus, 
$T^+$ is a semi-ample and big divisor on $Y$. 
Let 
\[
g\colon Y \to Z
\]
be the birational contraction induced by $T^+$ such that $g_*\MO_Y=\MO_Z$. 
It follows that $\Ex(g)=T^-$ and $g(\Ex(g))$ is one point. 
We define  $E^Z_i=g_*E^Y_i$ for $i=1, \dots, d$.

\begin{lem}\label{l-Z-Qfac}
$Z$ is $\Q$-factorial, $\rho(Z)=1$, and $-K_Z$ is ample. 
\end{lem}

\begin{proof}
Lemma~\ref{l-Qfac-criterion} and Lemma~\ref{l-Y-Qfac} imply that $Z$ is $\Q$-factorial and $\rho(Z)=1$.
 Corollary \ref{c-different} implies that  $-K_Y\sim_{\Q} T^++T^--\pi_Y^*K_T$ and, in particular, Lemma \ref{l_ST} implies that  $-K_Y$ is big. 
Thus,  $-K_Z$ is big. 
Since $\rho(Z)=1$, it follows that $-K_Z$ is ample. 
\end{proof}

\begin{lem}\label{l-Z-plt}
Fix $1\leq i \leq d$. 
Then the following  hold: 
\begin{enumerate}
\item 
\[
K_Y+E^Y_i+\frac{\psi^*A \cdot E_i-\frac{1}{2d-4}}{\psi^*A \cdot E_i}T^-=g^*(K_Z+E^Z_i).
\]
\item 
The pair $(Z, E^Z_i)$ is plt. 
\end{enumerate}
\end{lem}

\begin{proof}
We show (1). 
Let $b$ be the rational number satisfying 
\[
K_Y+E^Y_i+bT^-=g^*(K_Z+E^Z_i).
\]
By (4) and (5) of Lemma \ref{l-XY}, we have  
\[
K_Y \cdot f(E^{-}_i)=\frac{m_{\Gamma}-1}{m_{\Gamma}}+\psi^*A \cdot E_i+\frac{-\Gamma^2-2m_{\Gamma}}{-\Gamma^2 \cdot m_{\Gamma}}-1
\]
and 
\[
E^Y_i \cdot f(E^{-}_i)=\frac{1}{2d-4}.
\]
Since
\[
T^- \cdot f(E^-_i)=f^*(T^-) \cdot E^-_i=S^- \cdot E^-_i=-\psi^*A \cdot E_i,
\]
it follows that 
\[\begin{aligned}
b&=\frac{\frac{m_{\Gamma}-1}{m_{\Gamma}}+\psi^*A \cdot E_i+\frac{-\Gamma^2-2m_{\Gamma}}{-\Gamma^2 \cdot m_{\Gamma}}-1+\frac{1}{2d-4}}{\psi^*A \cdot E_i}\\
&=\frac{\psi^*A \cdot E_i-\frac{1}{2d-4}}{\psi^*A \cdot E_i}.
\end{aligned}
\]
Thus, (1) holds. (2) follows from (1) and Lemma~\ref{l-Y-discrep}. 
\end{proof}

\subsection{Some exact sequences}

\subsubsection{Cokernel of $\MO_X(-(n+1)S^-) \to \MO_X(-nS^-)$}

\begin{prop}\label{p-exact-X}
For any integer $n$, we have the exact sequence: 
\[
0 \to \MO_X(-(n+1)S^-) \to \MO_X(-nS^-) \to j_*\MO_S(\llcorner n\psi^*A\lrcorner) \to 0
\]
where $j\colon S \xrightarrow{\simeq} S^- \hookrightarrow X$ is the induced morphism. 
\end{prop}

\begin{proof}
Recall that $\widetilde {S}^-$ denotes the strict transform of $S^{-}$ in $\widetilde X$. Let $\widetilde j\colon \widetilde {S}^-\hookrightarrow \widetilde X$ be the induced morphism and 
consider the exact sequence 
\[
\begin{aligned}
0 \to \MO_{\widetilde X}(\ulcorner-(n+1)\mu^*S^-\urcorner)
&\to \MO_{\widetilde X}(\ulcorner-(n+1)\mu^*S^-\urcorner+\widetilde {S}^-) \\
&\to {\widetilde j}_*\MO_{\widetilde S^-}(\ulcorner-(n+1)\mu^*S^-\urcorner+\widetilde {S}^-) \to 0.
\end{aligned}
\]
It is enough to show the following: 
\begin{enumerate}
\item $\mu_*\MO_{\widetilde X}(\ulcorner-(n+1)\mu^*S^-\urcorner)=\MO_X(-(n+1)S^-)$, 
\item $\mu_*\MO_{\widetilde X}(\ulcorner-(n+1)\mu^*S^-\urcorner+\widetilde {S}^-)=\MO_X(-nS^-)$, 
\item $\MO_{\widetilde S^-}(\ulcorner-(n+1)\mu^*S^-\urcorner+\widetilde {S}^-) \simeq \MO_S(\llcorner n\psi^*A\lrcorner)$, and 
\item $R^1\mu_*\MO_{\widetilde X}(\ulcorner-(n+1)\mu^*S^-\urcorner)=0$. 
\end{enumerate}
(1) is clear. 
We now show (2). It is enough to prove the following  
{\small 
\begin{equation}\label{e-X-exact}
\MO_{\widetilde X}(\llcorner-n\mu^*S^-\lrcorner)
\subset 
\MO_{\widetilde X}(\ulcorner-(n+1)\mu^*S^-\urcorner+\widetilde {S}^-)
\subset 
\MO_{\widetilde X}(\ulcorner-n\mu^*S^-\urcorner).
\end{equation}
}
By (5) of Lemma~\ref{l-resolution}, we have that 
{\small 
\begin{equation}\label{e-X-exact2}
\mu^*S^-=\widetilde {S}^-+\sum_{C \subset \Ex(\psi)}(\frac{m_C-1}{m_C}F^-_{1}+\frac{m_C-2}{m_C}F^-_{2}+\cdots
+\frac{1}{m_C}F^-_{m_C-1}).
\end{equation}}
Then (\ref{e-X-exact}) follows from the fact that for any positive integers $m$ and $j$ with $1\leq j \leq m-1$, 
the following inequalities 
\[
\llcorner-n\frac{j}{m}\lrcorner \leq \ulcorner -(n+1)\frac{j}{m}\urcorner 
\leq \ulcorner -n\frac{j}{m}\urcorner
\]
hold. Thus, (2) holds. 

We now show (3). 
We may write 
\[
\MO_{\widetilde {S}^-}(\ulcorner-(n+1)\mu^*S^-\urcorner+\widetilde {S}^-) \simeq \MO_S(\llcorner n\psi^*A\lrcorner+\sum_{C \subset \Ex(\psi)}x_{n, C}C)
\]
for some integers $x_{n, C}$. 
Fix $C \subset \Ex(\psi)$. 
Let $\widetilde C^-$ be the curve corresponding to $C$ in $\widetilde S^-$. 
It suffices to show that $x_{n, C}=0$ for any integer $n$. 
Since $x_{n, C}$ is uniquely determined by taking the intersection number with $\widetilde C^-$, 
we may assume that $0 \leq n \leq m_C-1$. 
By (\ref{e-X-exact2}), we have  
\begin{eqnarray*}
&&(\ulcorner -(n+1)\mu^*S^-\urcorner+\widetilde {S}^-) \cdot \widetilde C^-\\
&=&(-(n+1)\mu^*S^-+\widetilde {S}^-+\frac{m_C-n-1}{m_C}F^-_1) \cdot \widetilde C^-\\
&=&(-n\mu^*S^--\frac{m_C-1}{m_C}F^-_1+\frac{m_C-n-1}{m_C}F^-_1) \cdot \widetilde C^-\\
&=&(-n\mu^*S^--\frac{n}{m_C}F^-_1) \cdot \widetilde C^-\\
&=&(n\psi^*A-\frac{n}{m_C}C) \cdot C\\
&=&\llcorner n\psi^*A \lrcorner \cdot C,
\end{eqnarray*}
where the second last equality holds by Lemma \ref{l-normal-bdl} and the 
last equality follows from  Assumption \ref{a-A}.
Thus, $x_{n, C}=0$ and (3) holds. 

(3) of Lemma \ref{l-resolution} and
(1) of Lemma~\ref{l-tildeX-kvv}   imply (4).
\end{proof}

\subsubsection{Exact sequences on $Y$}

\begin{prop}\label{p-Y-cokernel}
For any integer $n$, we have the exact sequence: 
\[
0 \to \MO_Y(-(n+1)T^-) \to \MO_Y(-nT^-) \to j'_*\MO_T(nA) \to 0
\]
where $j'\colon T \xrightarrow{\simeq} T^- \hookrightarrow Y$.
\end{prop}

\begin{proof}
By  Proposition~\ref{p-exact-X}, we have the exact sequence 
\[
0 \to \MO_X(-(n+1)S^-) \to \MO_X(-nS^-) \to j_*\MO_S(\llcorner n\psi^*A\lrcorner) \to 0.
\]
We apply $f_*$. 
By (5) of Lemma \ref{l-RC}, we have that  $S^-$ is $f$-numerically trivial. Thus,  we have  
\[
f_*\MO_X(-(n+1)S^-)=\MO_Y(-(n+1)T^-),\quad f_*\MO_X(-nS^-)=\MO_Y(-nT^-)
\]
and, Lemma~\ref{l-f-kvv} and Lemma~\ref{l-f-log-fano} imply 
\[
R^1f_*\MO_X(-(n+1)S^-)=0.
\]
We have
\[
f_*j_*\MO_S(\llcorner n\psi^*A\lrcorner)
=j'_*\psi_*\MO_S(\llcorner n\psi^*A\lrcorner)=j'_*\MO_T(nA).\]
Thus, the claim follows.
\end{proof}

\begin{prop}\label{p-Y-cokernel2}
For any integers $n$ and $j$ such that $1 \leq j \leq d$, 
 we have the exact sequence: 
\[
0 \to \MO_Y(-(n+1)T^--E^Y_j) \to \MO_Y(-nT^--E^Y_j) \to j'_*\MO_T(nA-E^T_j) \to 0
\]
where $j'\colon T \xrightarrow{\simeq} T^- \hookrightarrow Y$. 
\end{prop}

\begin{proof}
Since $E^X_j$ is Cartier, by  Proposition~\ref{p-exact-X}, we have the exact sequence 
\[
0 \to \MO_X(-(n+1)S^--E^X_j) \to \MO_X(-nS^--E^X_j) \to j_*\MO_S(\llcorner n\psi^*A\lrcorner-E_j) \to 0.
\]
We apply $f_*$. By (5) of Lemma \ref{l-RC}, we have that  $S^-$ is $f$-numerically trivial. 
Thus, $mS^-+E^X_j$ is $f$-nef for any integer $m$ and, in particular,
\[
f_*\MO_X(-mS^--E^X_j)=\MO_Y(-mT^--E^Y_j).
\]
Thus, it is enough  to prove the following
\begin{enumerate}
\item[(i)] $\psi_*\MO_S(\llcorner n\psi^*A\lrcorner-E_j)=\MO_T(nA-E^T_j)$. 
\item[(ii)] $R^1f_*\MO_X(-(n+1)S^--E^X_j)=0$. 
\end{enumerate}
We first show (i). 
Since 
\[
\psi^*E^T_j=E_j+\frac{1}{2}(\ell_j+\ell'_j)+\frac{1}{2d-4}\Gamma,
\]
we have that $-E_j=\ulcorner -\psi^*E^T_j\urcorner$. 
Since 
\[
\llcorner \psi^*(nA-E^T_j) \lrcorner \leq \llcorner n\psi^*A\lrcorner+\ulcorner -\psi^*E^T_j\urcorner \leq \ulcorner \psi^*(nA-E^T_j) \urcorner,\]
 (i) holds. 

We now show (ii). 
Since  $S^-$ is $f$-numerically trivial and,  by (1) of Lemma \ref{l-RC}, we have
\[
\begin{aligned}
f^*E^Y_j&=f^*\pi_Y^*E^T_j=\pi^*(E_j+\frac{1}{-\ell_j^2}\ell_j+\frac{1}{-\ell'^2_j}\ell'_j+\frac{1}{-\Gamma^2}\Gamma)\\
&=E^X_j+\frac{m_{\ell_j}}{-\ell_j^2}R_{\ell_j}+\frac{m_{\ell'_j}}{-\ell'^2_j}R_{\ell'_j}+\frac{m_{\Gamma}}{-\Gamma^2}R_{\Gamma},
\end{aligned}
\]
 Lemma~\ref{l-Y-discrep} implies that 
\[
-(n+1)S^--E^X_j-K_X \equiv_f 
\frac{m_{\ell_j}}{-\ell_j^2}R_{\ell_j}+\frac{m_{\ell'_j}}{-\ell'^2_j}R_{\ell'_j}+\frac{m_{\Gamma}}{-\Gamma^2}R_{\Gamma}+\sum_{C \subset \Ex(\psi)}\frac{-C^2-2m_C}{-C^2}R_C.
\]
Since each $-R_C$ is $f$-nef, 
we can find rational numbers $q_C$ such that $0 \leq q_C<1$ and 
$-(n+1)S^--E^X_j-(K_X+\sum_{C \subset \Ex(\psi)}q_CR_C)$ is $f$-nef. 
Since, by Lemma~\ref{l-X-plt},  $(X, \sum_{C \subset \Ex(\psi)}q_CR_C)$ is klt, 
  Lemma~\ref{l-f-kvv} implies (ii). 
\end{proof}

\subsection{Summary of notation}\label{ss-s-notation}

For the reader's convenience, we  now summarise some of the notation we introduced in this Section. 
Let  $\psi\colon S \to T$ be as in Subsection~\ref{ss-KM-surface}. 
We have a commutative diagram 
\[
\begin{CD}
\widetilde X @>\mu >>X @>f>> Y @>g>> Z\\
@. @VV\pi V @VV\pi_Y V\\
@. S @>\psi>> T,
\end{CD}
\]
where the upper maps are  birational morphisms of threefolds. 
\begin{itemize}
\item $S^{\pm}$ are sections of $\pi$ (cf. (\ref{n-negative}), (\ref{n-positive})). $T^{\pm}=f_*S^{\pm}$. 
\item $R_C=\pi^{-1}(C)_{\red}$ for any curve $C \subset \Ex(\psi)$. 
\item $E^X_i=\pi^{-1}(E_i)$. $E^Y_i=f_*E^X_i$. $E^Z_i=g_*E^Y_i$. 
\item $C^{\pm}$ is a curve on $X$ set-theoretically equal to $R_C \cap S^{\pm}$. 
\item $E_i^{\pm}$ is a curve on $X$ set-theoretically equal to $E^X_i \cap S^{\pm}$. 
\item $\widetilde D$ is the proper transform on $\widetilde X$ of a prime divisor $D$ 
on $X$. 
\item $F^+$ and $F^-_i$ are $\mu$-exceptional prime divisors 
(cf. Subsection~\ref{ss-resolution}). 
\item $A$ is an ample $\Z$-divisor on $T$ satisfying Assumption~\ref{a-A}. For any curve $C\subset \Ex(\psi)$, $m_C$ is a positive integer 
introduced in Assumption~\ref{a-A}. 
\end{itemize}

\section{Examples}\label{s-examples}

The goal of this Section is to prove our main results. 
Theorem~\ref{intro-plt} and Theorem~\ref{intro-extension} 
will be proved in Subsection~\ref{ss-plt}, whilst 
Theorem~\ref{intro-non-rational} and Theorem~\ref{intro-bad-Fano}
 in Subsection~\ref{ss-bad-Fano}. 

\subsection{Non-normal plt centres}\label{ss-plt}

\begin{nota}\label{n-plt}
We use the same notation as in Subsection~\ref{ss-KM-surface} and 
Section~\ref{s-cone} (cf. Subsection~\ref{ss-s-notation}). 
We fix a positive integer $q$ which satisfies the following properties:
\begin{enumerate}
\item $q \geq  2$. 
\item $d \geq q+2$. 
\item $2d-4$ is divisible by $q-1$ (e.g. $(d, q)=(5, 3)$). 
\end{enumerate}
Let 
\[
A:=\sum_{i=1}^qE^T_i-E^T_{q+1}.
\]
We have 
\[
\begin{aligned}
\psi^*A&=\sum_{i=1}^{q}(E_i+\frac{1}{2}(\ell_i+\ell'_i))-
(E_{q+1}+\frac{1}{2}(\ell_{q+1}+\ell'_{q+1}))+\frac{q-1}{2d-4}\Gamma\\
&=(\sum_{i=1}^{q}E_i-E_{q+1}-\ell_{q+1}-\ell'_{q+1})+ \frac{1}{2}\sum_{i=1}^{q+1} (\ell_i+\ell'_i)+\frac{q-1}{2d-4}\Gamma
\end{aligned}
\]
Thus, Assumption~\ref{a-A} is satisfied. 
\end{nota}

\begin{lem}\label{l-plt-compute}
With the same notation as above, the following  hold:
\begin{enumerate}
\item $H^1(T, \MO_T(nA))=0$ for $n \geq 0$. 
\item $H^1(T, \MO_T(A-E^T_{q+2}))\neq 0$. 
\item $H^2(T, \MO_T(nA-E^T_{q+2}))=0$ for $n \geq 0$. 
\end{enumerate}
\end{lem}

\begin{proof}
We first show (1). 
If $n=0$, then (1) follows from the fact that $T$ is a rational surface. 
If $n=1$, then (1) follows from (5) of  Theorem~\ref{t-KM-h1}. 
Thus, we may assume that $n \geq 2$. 
By (3) of Lemma \ref{l_ST}, we have that  $2E^T_i \sim 2E^T_j$ for any $i,j=1,\dots,d$. Thus, 
\[
2A =2\sum_{i=1}^qE^T_i-2E^T_{q+1} \sim 2(q-1)E^T_{q+1}.
\]
In particular, it follows that 
\[
3A \sim \sum_{i=1}^qE^T_i+(2q-3)E^T_{q+1}.
\]
Therefore, $nA$ is linearly equivalent to an effective ample divisor and, by (3) of Lemma \ref{l_ST},  so is $nA-K_T$. 
Thus, Proposition~\ref{p-eff-nef-big} implies that $H^1(T, \MO_T(K_T-nA))=0$ and,  Lemma \ref{l-CM} implies  (1). 

(5) of Theorem~\ref{t-KM-h1} implies (2), whilst  Lemma~\ref{l-CM} and (3) of Lemma \ref{l_ST} imply (3). 
\end{proof}

\begin{thm}\label{t-plt-main}
With the same notation as above, $E^Z_{q+2}$ is not normal. 
\end{thm}

\begin{proof}
The closed immersion  $i\colon E^Y_{q+2}\hookrightarrow Y$ induces  the exact sequence 
\[
0 \to \MO_Y(-E^Y_{q+2}) \to \MO_Y \to i_*\MO_{E^Y_{q+2}} \to 0.
\]
By applying $g_*$, we obtain the exact  sequence
\[
0 \to g_*\MO_Y(-E^Y_{q+2}) \to g_*\MO_Y \to g_*i_*\MO_{E^Y_{q+2}} 
\to R^1g_*\MO_Y(-E^Y_{q+2}) \to R^1g_*\MO_Y.
\]
If $E^Z_{q+2}$ is normal, then the natural homomorphism 
$g_*\MO_Y \to g_*i_*\MO_{E^Y_{q+2}}$ is surjective. 
Thus, it is enough to show that 
\[
g_*\MO_Y \to g_*i_*\MO_{E^Y_{q+2}}
\]
is not surjective. 
To this end, it is enough to prove the following  
\begin{enumerate}
\item $R^1g_*\MO_Y =0$. 
\item $R^1g_*\MO_Y(-E^Y_{q+2}) \neq 0$.
\end{enumerate}

We first show (1). 
For any $n \geq 0$, by Proposition~\ref{p-Y-cokernel} we have the exact sequence 
\[
0 \to \MO_Y(-(n+1)T^-) \to \MO_Y(-nT^-) \to j'_*\MO_T(nA) \to 0,
\]
which induces the exact sequence
\[
R^1g_*\MO_Y(-(n+1)T^-) \to R^1g_*\MO_Y(-nT^-) \to H^1(T, \MO_T(nA)).
\]
(1) of Lemma~\ref{l-plt-compute} implies that   $H^1(T, \MO_T(nA))=0$ for any $n\geq 0$. Thus,   Serre vanishing implies (1). 

We now show (2). 
For any $n \geq 0$, by Proposition~\ref{p-Y-cokernel2}, we have the exact sequence 
\[
0 \to \MO_Y(-(n+1)T^--E^Y_{q+2}) \to \MO_Y(-nT^--E^Y_{q+2}) \to j'_*\MO_T(nA-E^T_{q+2}) \to 0.
\]
We apply $g_*$. For $n=0$, we have  the exact sequence 
\[
0=H^0(T, \MO_T(-E^T_{q+2})) \to R^1g_*\MO_Y(-T^--E^Y_{q+2}) \to R^1g_*\MO_Y(-E^Y_{q+2}).
\]
Thus, it suffices to show that $R^1g_*\MO_Y(-T^--E^Y_{q+2}) \neq 0.$ 

For $n=1$, we have the exact sequence 
\[
R^1g_*\MO_Y(-T^--E^Y_{q+2}) \to H^1(T, \MO_T(A-E^T_{q+2})) \to R^2g_*\MO_Y(-2T^--E^Y_{q+2}).
\]
 (2) of Lemma~\ref{l-plt-compute} implies that  $H^1(T, \MO_T(A-E^T_{q+2})) \neq 0$. Thus,  
it is enough to show that $R^2\MO_Y(-2T^--E^Y_{q+2})=0$. 

(3) of Lemma~\ref{l-plt-compute} implies that $H^2(T, \MO_T(nA-E^T_{q+2}))=0$.
For any $n\ge 2$, we consider the  exact sequence 
\[
R^2g_*\MO_Y(-(n+1)T^--E^Y_{q+2}) \to R^2g_*\MO_Y(-nT^--E^Y_{q+2}) \to H^2(T, \MO_T(nA-E^T_{q+2}))=0.
\]
 Since $R^2g_*\MO_Y(-nT_--E^Y_{q+2})=0$ for any  $n \gg 0$, the claim follows. Thus, (2) holds. 
\end{proof}

\begin{rem}
By \cite[Proposition~3.11]{GNT15},
the normalisation of $E^Z_{q+2}$ in Theorem~\ref{t-plt-main} 
is a universal homeomorphism. 
\end{rem}

\begin{proof}[Proof of Theorem~\ref{intro-plt}]
Using the same notation as above, let $E=E^Z_{q+2}$. 
Then the claim follows from Lemma~\ref{l-Z-plt} and Theorem~\ref{t-plt-main}. 
\end{proof}

\begin{proof}[Proof of Theorem~\ref{intro-extension}]
We use Notation~\ref{n-plt}. Recall that $g\colon Y \to Z$ is a birational morphism such that   $\Ex(g)=T^{-}$. 
By (2) of Lemma~\ref{l-Z-plt}, the pair $(Z, E^Z_{q+2})$ is plt and
if 
\[
B=\frac {q-2}{q-1}T^-,
\]
then,  by (1) of Lemma~\ref{l-Z-plt}, 
we may write 
 \[
K_Y+E^Y_{q+2}+B=g^*(K_Z+E^Z_{q+2}).
\]
Thus, the pair $(Y, E^Y_{q+2}+B)$ is plt.

As in  the proof of Theorem~\ref{t-plt-main}, 
the induced morphism 
\[
g_*\MO_Y \to g_*i_*\MO_{E^Y_{q+2}}
\]
is not surjective. 

Let $m_0$ be a positive integer such that $m_0(K_Z+E_{q+2}^Z)$ is Cartier. 
Given $P:=g(T^-)$, let $Z'$ be an affine open subset containing $P$ such that 
\[
\MO_Z(m_0(K_Z+E_{q+2}^Z))|_{Z'} \simeq \MO_{Z'}.
\]
Let $Y'$  be the inverse image of $Z'$ and 
\[
E':=E^Y_{q+2}|_{Y'},\quad  B':=B|_{Y'}.
\]
Then the claim follows.
\end{proof}

\subsection{A klt Fano threefold $X$ with $H^2(X,\MO_X) \neq 0$}\label{ss-bad-Fano}


\begin{nota}\label{n-bad-Fano}
We use the same notation as in Subsection~\ref{ss-KM-surface} and 
Section~\ref{s-cone} (cf. Subsection~\ref{ss-s-notation}). 
We fix a positive integer $q$ and we define 
$d:=4q+2$.  Let 
\[
A:=\sum_{i=1}^{3q}E^T_i-\sum_{j=3q+1}^{4q}E^T_j.
\]
We have 
\[
\begin{aligned}
\psi^*A&=\sum_{i=1}^{3q}(E^T_i+\frac{1}{2}(\ell_i+\ell'_i))-
\sum_{j=3q+1}^{4q}(E^T_j+\frac{1}{2}(\ell_j+\ell'_j))+\frac{1}{4}\Gamma\\
&=(\sum_{i=1}^{3q}E^T_i - \sum_{j=3q+1}^{4q} (E^T_j +\ell_j+\ell'_j)) + \frac 1 2 \sum_{i=1}^{4q}(\ell_i+\ell'_i)+\frac 1 4 \Gamma. 
\end{aligned}
\]
Thus, Assumption~\ref{a-A} is satisfied. 
\end{nota}

\begin{lem}\label{l-h1-bad-Fano}
With the same notation as above, the following  hold: 
\begin{enumerate}
\item $h^1(T, \MO_T(A))=q-1$. 
\item $H^i(T, \MO_T(nA))=0$ for $i >0$ and $n \geq 2$. 
\end{enumerate}
\end{lem}

\begin{proof} (1) follows from (5) of Theorem~\ref{t-KM-h1}. 

We now show (2). 
If $i=2$, then  Lemma~\ref{l-CM} implies (2).
Thus, we may assume that $i=1$. 
By (3) of Lemma \ref{l_ST}, we have that  $2E^T_i \sim 2E^T_j$ for any $i,j=1,\dots,d$. Thus, 
\[
H^0(T, \MO_T(2A)) \neq 0 \quad \text{and}\text \quad H^0(T, \MO_T(3A)) \neq 0.
\]
In particular, it follows that  $H^0(T, \MO_T(nA)) \neq 0$ for any $n \geq 2$. 
Thus,  $nA-K_T$ is linearly equivalent to an effective ample divisor and Proposition~\ref{p-eff-nef-big}
implies that $H^1(T, \MO_T(nA))=0$.
Thus, (2) holds. 
\end{proof}

\begin{thm}\label{t-bad-Fano}
With the same notation as above, 
$Z$ is a projective $\Q$-factorial klt threefold which satisfies the following 
properties: 
\begin{enumerate}
\item $-K_Z$ is ample and $\rho(Z)=1$, 
\item $H^1(Z, \omega_Z)=0$, 
\item $\dim_k H^2(Z, \MO_Z)=q-1$,   
\item there exists a birational morphism $h\colon W \to Z$ 
from a smooth projective threefold $W$ such that 
$\dim_k H^0(Z, R^1g_*\MO_W)=q-1$, and 
\item if $q\geq 2$, then $Z$ is not Cohen--Macaulay. 
\end{enumerate}
\end{thm}

\begin{proof}
By Lemma~\ref{l-Z-Qfac} and Lemma~\ref{l-Z-plt}, 
$Z$ is a projective $\Q$-factorial klt threefold which satisfies  (1). 
Let $W:=\widetilde X$ (cf. Subsection \ref{ss-resolution}) and let 
\[
h\colon W=\widetilde X \xrightarrow{\mu} X \xrightarrow{f} Y \xrightarrow{g} Z.
\]

We now show (2). 
Since $Z$ is klt, we can write $K_{W}+E_1=h^*K_Z+E_2$ 
for some effective $h$-exceptional $\Q$-divisors $E_1$ and $E_2$ 
such that $\llcorner E_1 \lrcorner =0$. 
This implies that 
\[
h_*\MO_W(K_W)=
h_*\MO_W(K_W+E_1)=
h_*\MO_W(h^*K_Z+E_2)=\MO_Z(K_Z),
\]
where the first equality follows from $\llcorner E_1 \lrcorner =0$. 
Thus, we have that 
\begin{equation}\label{e-omega}
h_*\omega_W=\omega_Z.
\end{equation}
By the Leray spectral sequence, we have an injection 
\begin{equation}\label{e-Leray-inje}
H^1(Z, h_*\omega_W) \hookrightarrow H^1(W, \omega_W).
\end{equation}
Since $W$ is a smooth rational variety, it follows that $H^1(W, \omega_W)=0$ 
by \cite[Theorem 1]{CR11}. 
Thus, (\ref{e-omega}) and (\ref{e-Leray-inje}) imply (2). 
Note that (5) follows from (2), (3) and \cite[Theorem 5.71]{km98}. 

In order to prove (3) and (4), it is enough  to show the following: 
\begin{enumerate}
\item[(i)] $R^1h_*\MO_W \simeq R^1g_*\MO_Y$. 
\item[(ii)] $H^0(Z, R^1g_*\MO_Y) \simeq H^2(Z, \MO_Z).$
\item[(iii)] $h^0(Z, R^1g_*\MO_Y)=q-1$.
\end{enumerate}
 (i) and (ii) follow from Lemma~\ref{l-Y-rational}. 

We now show (iii). 
By Proposition~\ref{p-Y-cokernel}, We have the exact sequence 
\[
0 \to \MO_Y(-(n+1)T^-) \to \MO_Y(-nT^-) \to j'_*\MO_T(nA) \to 0
\]
for any $n \geq 0$. 
By (2) of Lemma~\ref{l-h1-bad-Fano} and  Serre vanishing theorem, 
we have that $R^ig_*\MO_Y(-nT^-)=0$ for $i>0$ and $n \geq 2$. 
Thus, if $n=1$, we have 
\begin{equation}\label{e-bad-Fano}
R^1g_*\MO_Y(-T^-) \simeq H^1(T, \MO_T(A)).
\end{equation}
Let $n=0$. By Lemma \ref{l-T-normal}, it follows that 
 $H^1(T^-, \MO_{T^-})=0$ and $g_*\MO_Y \to g_*\MO_{T^-}$ is surjective. Thus, 
\begin{equation}\label{e-bad-Fano2}
R^1g_*\MO_Y(-T^-) \simeq R^1g_*\MO_Y.
\end{equation}
 By (1) of Lemma~\ref{l-h1-bad-Fano},  we have that  $\dim_k H^1(T, \MO_T(A))=q-1$. Thus, 
(\ref{e-bad-Fano}) and (\ref{e-bad-Fano2}) imply (iii). 
\end{proof}

\begin{proof}[Proof of Theorem~\ref{intro-non-rational} and 
Theorem~\ref{intro-bad-Fano}]
Both Theorems follow directly from Theorem~\ref{t-bad-Fano}. 
\end{proof}

\bibliographystyle{amsalpha}
\bibliography{Library}

\end{document}